\documentclass{amsart}
\usepackage{amssymb}
\usepackage{color}
\usepackage{graphicx}
\usepackage[all]{xy}
\usepackage{hyperref}

\usepackage{ulem}
\newtheorem{thm}{Theorem}[section]
\newtheorem{lem}[thm]{Lemma}
\newtheorem{prop}[thm]{Proposition}
\newtheorem{cor}[thm]{Corollary}
\theoremstyle{definition}
\newtheorem{defn}[thm]{Definition}
\newtheorem{rmk}[thm]{Remark}
\newtheorem{ex}{Example}

\newcommand{\stkout}[1]{\ifmmode\text{\sout{\ensuremath{#1}}}\else\sout{#1}\fi}

\def\K{\mathcal{K}}
\def\LK{\mathcal{LK}}
\def\SK{\mathcal{SK}}
\def\LSK{\mathcal{LSK}}
\def\R{\mathbb{R}}
\def\Z{\mathbb{Z}}
\def\P{\mathcal{P}}

\begin{document}
	\title{Legendrian singular links and singular connected sums}
	\author{Byung Hee An}
	\address{Center for Geometry and Physics, Institute for Basic Science (IBS), Pohang 790-784, Republic of Korea}
	\email{anbyhee@ibs.re.kr}
	\author{Youngjin Bae}
	\email{yjbae@ibs.re.kr}
	\author{Seonhwa Kim}
	\email{ryeona17@ibs.re.kr}
	
	\begin{abstract}
		We study Legendrian singular links up to contact isotopy. Using a special property of the singular points, we define the singular connected sum of 
		Legendrian singular links. This concept is a generalization of the connected sum and can be interpreted as a tangle replacement,  
		which provides a way to classify Legendrian singular links. Moreover, we investigate several phenomena only occur in the Legendrian setup.
	\end{abstract}
	\maketitle
	\tableofcontents
	\section{Introduction}
	A Legendrian singular link of degree $m$ with $n$-components is the image of an immersion of $n$-copies of $S^1$ into $S^3$ whose tangent vectors are contained in the contact structure $(S^3,\xi_{std})$ {and} which has $m$ transverse double points as its only singularities.
	Legendrian singular links are discussed in \cite{FT,T} as a theme of Vassiliev type invariants, and appeared in \cite{Ch} to give an algorithm for producing possible Lagrangian projections of Legendrian knots.
	To the best of the authors' knowledge, Legendrian singular links have not yet been studied in their own right.
	
	The $h$-principle \cite[\S16.1]{EM} says that the study of Legendrian singular links up to Legendrian regular homotopy reduces to a homotopic theoretic question,
	thus there can be no interesting phenomena from the perspective of contact topology.
	We instead study Legendrian singular links up to (ambient) contact isotopy, which preserves transversality\footnote{This is not to be confused with the transverse knots. Here `transverse' means that the two tangent vectors at the singular point span the contact plane at that point.} and the Legendrian property at each singular point. 
	
	The degree of a given Legendrian singular link can be reduced via {\em resolutions\footnote{Sometimes called `smoothing' in the literature.}} as usual for singular links.
	So Legendrian singular links ($\LSK$) can be reduced to singular links ($\SK$) via the {\em forgetful map $\|\cdot\|$}, which takes the underlying singular link type, and to Legendrian links ($\LK$) via resolutions $\mathcal{R}$ with the following commutative diagram of various link theories:
	$$
	\xymatrix{
		\mathcal{LSK}\ar[rr]^-{\mathcal{R}}\ar[d]_{\|\cdot\|}&&\mathcal{LK}\ar[d]^{\|\cdot\|}\\
		\mathcal{SK}\ar[rr]^{\mathcal{R}}&&\mathcal{K}
	}
	$$
	See \S\ref{sec:variouslinks} and \ref{sec:resolution} for the precise definitions.
	
	The goal of this article is twofold. First, we investigate various invariants for $\LSK$ including Thurston-Bennequin number, rotation number, and the resolutions with supporting examples and argue that $\LSK$ is not a straightforward combination of $\LK$ and $\SK$. 
	The other is to develop a useful tool, called {\em singular connected sum}, and show that it distinguishes a particular pair of Legendrian singular links that can not be distinguished in $\LK$ under any resolution or in $\SK$ under $\|\cdot\|$.
	
	The above two goals are deeply related to a special property of the singular points of Legendrian singular links.
	Specifically, through contact isotopy, one can keep track of the relative position of two tangent vectors at each singular point by the co-orientation of the contact structure $\xi_{std}$ on $S^3$.
	This allows to define an {\em order} at each singular point which is equivariant under contact isotopy. 
	
	Moreover this property enables us to define the notion of connected sum at singular points. We define a {\em singular connected sum $(L_1,p_1)\otimes(L_2,p_2)$} by simultaneously performing connected sums on two pairs of arcs near singular points $p_i$ of $L_i$.
	
	\begin{thm}\label{thm:singularconnectedsum}
		For a given pair of Legendrian singular links $L_1,L_2$ with singular points $p_1$, $p_2$,
		the singular connected sum $(L_1,p_1)\otimes(L_2,p_2)$ is well-defined.
	\end{thm}
	
	\begin{thm}\label{thm:singularconnectedsumdecomposition}
		Let $L$ be a Legendrian singular link and $S$ be a separating sphere for $L$ inducing a decomposition
		$L=(L_1,p_1)\otimes(L_2,p_2)$.
		Then this decomposition is well-defined up to order-preserving contact isotopy of $S$ with respect to $L$.
	\end{thm}
It is worth remarking that neither the singular connected sum nor the decomposition are well-defined in $\SK$.
Moreover 
%
we have the following rigidity phenomenon that only occurs in $\LSK$. 
		\begin{thm}\label{thm:unit}
			With respect to the singular connected sum, there exists the identity element, and moreover the only unit of degree 2 is the identity itself.
%
		\end{thm}
It would be interesting to study more about algebraic aspects of the singular connected sum on $\LSK$.
	
	On the other hand, the singular connected sum is the same as the replacement of a singular point $p_1\in L_1$ with a specific singular Legendrian tangle obtained from $(L_2,p_2)$, and {\em vice versa}.
	Indeed, the idea of Legendrian tangles and their replacement is already discussed in the literature including \cite{NT,MS,S}, although their approaches are slightly different from ours.
	There is a diagrammatic interpretation of the singular connected sum as well, which allows us to handle the operation in a convenient way.
	This interpretation is related to the {\em vertical cut} of the front projection, discussed in \cite{S}.
	
	As an application of the singular connected sum, we have the following theorem which implies that $\LSK$ is more than the {\em pull-back} of $\LK$ and $\SK$ in the commutative diagram above.

	\begin{thm}\label{thm:pairoflsk}
		There exist two Legendrian singular links sharing all classical invariants, Legendrian link types of all resolutions, and invariants from the orders, which are not contact isotopic to one another.
	\end{thm}
	
	For a given $L\in\LSK$ of degree $k$ one can obtain a double $\mathcal{D}(L)$, a Legendrian link in $\#^{k-1}(S^2\times S^1)$, by a {\em multiple singular connected sum} of $L$ with itself. 
	Thanks to the work of \cite{EN} we can assign a Legendrian contact homology algebra of $\mathcal{D}(L)$ to $L$, as an algebraic invariant of $L$.
	
	Furthermore, the resolutions can be regarded as special cases of tangle replacements, and each resolution has a unique inverse operation, called a {\em splicing}, under certain splitting conditions.
	These splicings provide full descriptions of Legendrian singular links with certain singular link types. See Theorem~\ref{thm:simple} and Corollary~\ref{cor:Linftysimple}.
	
	\subsection*{Acknowledgement} 
	We are grateful to Gabriel C. Drummand-Cole for his valuable and detailed comments on a previous draft.
	This work was supported by Center for Geometry and Physics, Institute for Basic Science (IBS-R003-D1).
	
	\section{Preliminaries}
	
	\subsection{Legendrian singular links in \texorpdfstring{$S^3$}{3-sphere}}\label{sec:variouslinks}
	Throughout this paper, we regard $S^3$ as the unit sphere in $\mathbb{C}^2$. Then the {\em standard contact structure $\xi_{std}$ on $S^3$} is given by
	\begin{align*}
		\xi_{std}&=\ker\lambda_{std};\\
		\lambda_{std}&=r_1^2d\theta_1+r_2^2d\theta_2\\
		&=x_1dy_1-y_1dx_1+x_2dy_2-y_2dx_2,
	\end{align*}
	where $z=r_1e^{i\theta_1}=x_1+i y_1,w=r_2e^{i\theta_2}=x_2+i y_2$.
	
	For convenience's sake, we frequently consider $S^3=\R^3\cup\{\infty\}$ as the one-point compactification of $\R^3$ with two contact structures {$\xi_{rot}$ and $\xi_0$} which are contactomorphic and defined as follows.
	\begin{align*}
		\xi_{rot}=\ker\alpha_{rot},\quad&\alpha_{rot}=dz+r^2d\theta=dz+xdy-ydx;\\
		\xi_0=\ker\alpha_0,\quad&\alpha_0=dz-ydx.
	\end{align*}
	
	From now on, we assume that the contact structure on $S^3$ is always co-oriented by $\lambda_{std}$.
	
	We define Legendrian singular links and describe its relation to known knot theories.
	All types of links we will consider in this article are oriented unless otherwise stated.
	
	Let $nS^1$ be a disjoint union $\coprod_{i=1}^n S^1_i$ of $n$-copies of $S^1$.
	A {\em link $K$ with $n$-components} is the oriented image of a smooth embedding $nS^1 \hookrightarrow S^3$, and a {\em singular link $K_s$ of degree $m$} is a link defined by using an immersion instead of an embedding with precisely $m$ {\em transverse double points}, called {\em singular points}.
	We denote the set of singular points by $\P(K_s)$.
	
	Now we endow $S^3$ with the standard contact structure $\xi_{std}$ described above.
	A {\em Legendrian link $L$} is a link with every tangent vector lying in the contact structure $\xi_{std}$, and a {\em Legendrian singular link $L_s$} is a singular link with the same tangency condition.
	
	We say that two (Legendrian) links $K_0$, $K_1$ are {\em equivalent} if there exists a (contact) ambient isotopy $h_t:S^3\to S^3$ such that $h_0$ is the identity and $h_1(K_0)=K_1$. We call the equivalence class a {\em (Legendrian) link type}.
	We denote by $\K$ ($\LK$) and $\SK$ ($\LSK$) the collections of (Legendrian) link types and (Legendrian) singular link types, respectively.
	
	For $K_s\in\SK$, we denote by $\mathcal{L}(K_s)$ the set of all Legendrian singular links of singular link type $K_s$.
	Conversely, we denote by $\|L_s\|$ singular link type of $L_s$.
	
	\subsection{Projections}
	Since any Legendrian isotopy of a Legendrian immersion can be assumed not to touch a designated point $\infty$ in $S^3$, we may assume that Legendrian singular links lie in
	$(\R^3,\xi_0)$ as usual for $\mathcal{LK}$.
	
	The {\em front projection $\pi_F$} and {\em Lagrangian projection $\pi_L$} are defined as the projections of $(\R^3,\xi_0)$ onto the $xz$-plane and $xy$-plane, respectively, as follows.
	\begin{align*}
		\pi_F(x,y,z)=(x,z),\quad
		\pi_L(x,y,z)=(x,y).
	\end{align*}
	Note that we are able to recover $L$ from $\pi_L(L)$ up to a shift in the $z$-coordinate or $\pi_F(L)$ by using the Legendrian condition, and
	the projections near $p\in \P(L)$ look like `$\,\vcenter{\hbox{\includegraphics[scale=1]{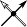}}}\,$' in the Lagrangian projection and `$\,\vcenter{\hbox{\includegraphics[scale=1]{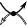}}}\,$' in the front projection. See Figure~\ref{fig:singularpoints}.
	For each singular point, we indicate a dot to avoid confusion with ordinary double points (crossings) in both the front and Lagrangian projections.
	
	\begin{figure}[ht]
		$$
		\begin{array}{c|c|c|c|c}
		\text{Front}&\;\;
		\vcenter{\hbox{\includegraphics[scale=1]{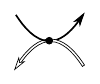}}}\;\; &\;\;
		\vcenter{\hbox{\includegraphics[scale=1]{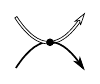}}}\;\; &\;\;
		\vcenter{\hbox{\includegraphics[scale=1]{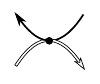}}}\;\; &\;\;
		\vcenter{\hbox{\includegraphics[scale=1]{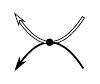}}} \\
		\hline
		\text{Lagrangian}&\;\;
		\vcenter{\hbox{\includegraphics[scale=1]{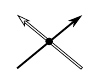}}}\;\;&\;\;
		\vcenter{\hbox{\includegraphics[scale=1]{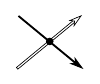}}}\;\;&\;\;
		\vcenter{\hbox{\includegraphics[scale=1]{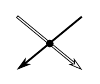}}}\;\;&\;\;
		\vcenter{\hbox{\includegraphics[scale=1]{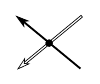}}}
		\end{array}
		$$
		\caption{Projections near a singular point}
		\label{fig:singularpoints}
	\end{figure}
	
	A {\em diagram $D\subset\R^2_{xz}$} consists of piecewise smooth closed curves in the $xz$-plane without a vertical tangency, which may have cusps with a smooth tangency condition\footnote{If $D$ is parametrized by $t$ and has a cusp at $t_0$, then $y(t_0)=\lim_{t\to t_0}\frac{z'(t)}{x'(t)}$ is well-defined and smooth near $t_0$.}.
	
	We assume further that every nontransversal double point $p$ in $D$ is parameterized like one of front projections depicted above.
	Then it is easy to see that any diagram $D$ can be realized by $\pi_F(L)$ for some $L\in\LSK$ and {\em vice versa}. Therefore we do not distinguish a front projection and a diagram unless any ambiguity occurs.
	
	A diagram $D$ is said to be {\em regular} if $D$ has no triple (or more) point, and none of its double points is a cusp.
	Note that in all possible diagrams, the set of regular diagrams are dense, and therefore for any $L\in\LSK$, we may assume that the front projection $\pi_F(L)$ is regular by perturbing $L$ slightly.
	Non-regular examples are shown in Figure~\ref{fig:forbidden}.
	
	\begin{figure}[ht]
		\centering
		\includegraphics{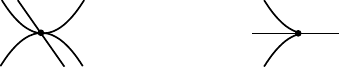}
		\caption{Examples of non-regular front projections near a singular point}
		\label{fig:forbidden}
	\end{figure}
	
	Moreover, for a given contact isotopy $\phi_t$ starting with $L$, the 1-parameter family $D_t$ of diagrams, defined by $D_t=\pi_F(\phi_t(L))$, is regular for all but finitely many $t$'s.
	Let $\{t_1,\dots,t_k\}$ be the set of such $t$'s such that at each $t_i$, there is exactly one point in $D_{t_i}$ violating the regularity.
	Then during $t_i<t<t_{i+1}$ for each $i$, the variance of $D_t$ can be regarded as the result under a plane isotopy on $\R_{xz}$, which does not produce the vertical tangency.
	On the other hand, the diagrams $D_{t_i-\epsilon}$ and $D_{t_i+\epsilon}$ for a small $\epsilon>0$ essentially differ by (a composite of) the {\em Reidemeister moves} depicted in Figure~\ref{fig:frontmove}.
	
	This follows from a result about Legendrian graphs in \cite{BI} by regarding $L$ as a Legendrian graph which has 4-valent vertices only and satisfies certain tangency conditions at each vertex.
	Conversely, at each vertex of valency 4, there is a canonical way to smooth edges and obtain two transverse arcs.
	Hence there is essentially no difference between Legendrian singular links and Legendrian 4-valent graphs.
	Note that the moves in Figure~\ref{fig:frontmove} are slightly different from those in \cite{BI} because we do not allow cusps to be double points.
	
	\begin{figure}[ht]
		\centering
		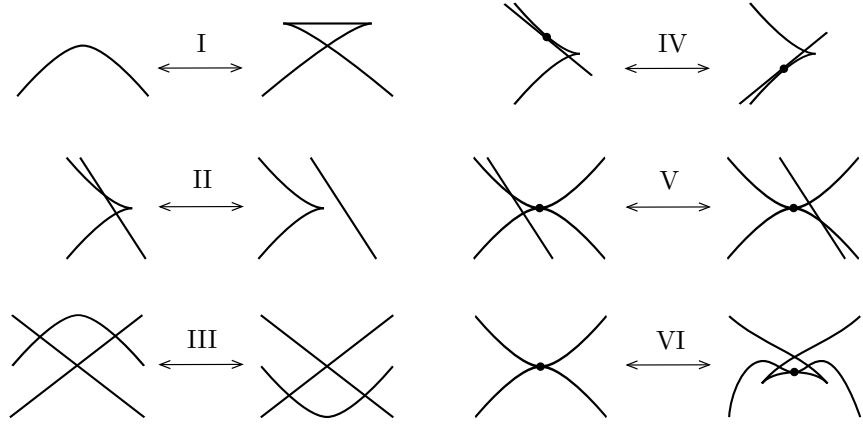
		\caption{Reidemeister moves for $\LSK$}
		\label{fig:frontmove}
	\end{figure}
	
	\begin{prop}\label{prop:frontmove}
		Let $L_1, L_2$ be Legendrian singular links. Then $L_1$ and $L_2$ are equivalent in $\LSK$ if and only if $\pi_F(L_1)$ and $\pi_F(L_2)$ are related by a sequence of plane isotopies and moves $({\rm I})\sim ({\rm VI})$ including their reflections about the $x$ and $z$-axes, depicted in Figure~\ref{fig:frontmove}.
	\end{prop}
	\begin{proof}
		We refer to the result for Legendrian graphs in \cite{BI}.
		Then two types of moves in \cite[Figure~9]{BI} involve the forbidden front projection as shown in Figure~\ref{fig:4a4b}.
		We call these moves {\em forbidden moves}.
		Note that $\rm (VI_*)$ is a composition of $\rm(VI)$ and $\rm(IV_*)$ for each $*\in\{a,b\}$. Hence we need not consider the moves $\rm(VI_*)$.
		\begin{figure}[ht]
			\centering
			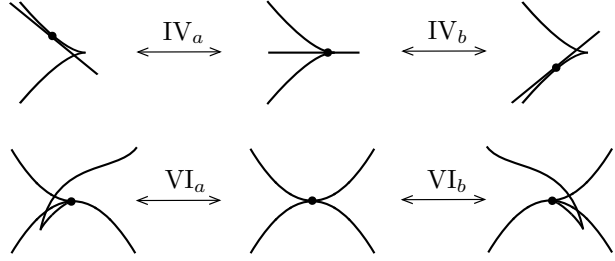	
			\caption{Moves $\rm (IV_*)$ and $\rm(VI_*)$}
			\label{fig:4a4b}
		\end{figure}
		
		Then we have the following lemmas which are easy observations whose proofs we omit.
		\begin{lem}\label{lem:forbiddenmove1}
			Let $D\stackrel{F_1}{\longrightarrow}D'\stackrel{F_2}{\longrightarrow}D''$ be a sequence of diagrams connected by forbidden moves $F_i\in\{\rm(IV_a),(IV_b)\}$ at singular points $p_i$.
			Then $F_2\circ F_1$ is either
			\begin{enumerate}
				\item the identity if $p_1=p_2$ and $F_1=F_2$;
				\item $\rm(IV)$ if $p_1=p_2$ but $F_1\neq F_2$;
				\item $F_1\circ F_2$ if $p_1\neq p_2$.
			\end{enumerate}
		\end{lem}
		
		\begin{lem}\label{lem:forbiddenmove2}
			Let $D\stackrel{F}{\longrightarrow}D'\stackrel{R}{\longrightarrow}D''$ be a sequence of diagrams connected by a forbidden move $F\in\{\rm(IV_a),(IV_b)\}$ at singular points $p$ and a regular move $R\in\{\rm(I),\dots,(VI)\}$.
			Suppose $D$ is regular near $p$.
			Then $D''$ is non-regular near $p$, and $R\circ F$ is $F\circ R$.
		\end{lem}
		
		Let $D_i=\pi_F(L_i)$. Then there is a sequence $\mathbf{R}$ of Reidemeister moves $\{\rm(I),\dots,(VI),(IV_a),(IV_b)\}$, which transforms $D_1$ into $D_2$.
		We use the induction on the number of forbidden moves in $\mathbf{R}$.
		
		Let $F_1$ be the first occurrence of a forbidden move in $\mathbf{R}$ involving a singular point $p$. 
		Then since both $D_1$ and $D_2$ are regular near $p$, there must be another occurrence of a forbidden move at $p$ in $\mathbf{R}$ after $F_1$.
		Let $F_2$ be the second one. Then by definition, there is no move involving a singular point $p$ between $F_1$ and $F_2$ in $\mathbf{R}$.
		Therefore by Lemma~\ref{lem:forbiddenmove1}~(3) and \ref{lem:forbiddenmove2}, $F_1$ moves forward in $\mathbf{R}$ until it meets $F_2$. Then by Lemma~\ref{lem:forbiddenmove1}~(1) or (2), they are cancelled or become a regular move $\rm(IV)$.
		Hence the number of forbidden moves decreases by 2, and the proposition follows by induction.
	\end{proof}
	
	\subsection{Resolutions}\label{sec:resolution}
	In $\SK$, a resolution is the standard way to reduce the number of singular points, and eventually to obtain nonsingular links.
	By virtue of Proposition~\ref{prop:frontmove}, links in $\LSK$ are described diagrammatically, and so are resolutions as follows.
	
	\begin{defn}
		Let $L\in\LSK$ and $p\in\P(L)$. For $\eta\in\{+,-,0,\infty\}$,
		an {\em $\eta$-resolution $R_\eta(L,p)$ of $L$ at $p$} is defined by replacing a small neighborhood of $p$ in $\pi_F(L)$ with the corresponding diagram $R_\eta$ depicted in Figure~\ref{fig:resolutions}.
	\end{defn}
	
	\begin{figure}[ht]
		\centering
		\begin{align*}
			\begin{array}{c|c|c|c|c}
				(L,p) & R_+(L,p) & R_-(L,p) & R_0(L,p) & R_\infty (L,p)\\
				\hline
				\vcenter{\hbox{\includegraphics[scale=1]{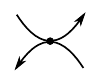}}}&
				\vcenter{\hbox{\includegraphics[scale=1]{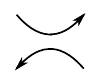}}}&
				\vcenter{\hbox{\includegraphics[scale=1]{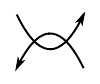}}}&
				\vcenter{\hbox{\includegraphics[scale=1]{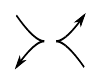}}}&
				\vcenter{\hbox{\includegraphics[scale=1]{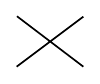}}}\\
				\hline
				\vcenter{\hbox{\includegraphics[scale=1]{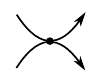}}}&
				\vcenter{\hbox{\includegraphics[scale=1]{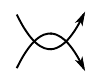}}}&
				\vcenter{\hbox{\includegraphics[scale=1]{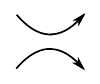}}}&
				\vcenter{\hbox{\includegraphics[scale=1]{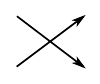}}}&
				\vcenter{\hbox{\includegraphics[scale=1]{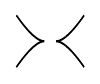}}}\\
			\end{array}
		\end{align*}
		\caption{$R_+, R_-, R_0$ and $R_\infty$ in the front projection}
		\label{fig:resolutions}
	\end{figure}
	
	The well-definedness for all $R_\eta$ under ambient isotopy also follows from the fact that 
	the push-forwards of each Reidemeister move $\rm (I)\sim(VI)$ along any resolution $R_\eta$ are reduced to (sequences of) Reidemeister moves ${\rm (I)\sim(III)}$.
	Therefore $R_\eta$ does not depend on the diagram but only on the $\LSK$ type.
	
	Since the resolutions $R_\pm$ and $R_0$ are local replacement of oriented diagrams, the order of taking these resolutions does not matter. 
	Note that $R_\pm$ preserves the number of components but $R_0$ increases or decreases the number of components by $1$.
	On the other hand, $R_\infty$ does not induce an orientation and that it may preserve the number of components or decrease it by $1$.
	
	Let $\mathcal{R}(L)\subset\LK$ be the set of {\em full resolutions} consisting of Legendrian nonsingular links obtained from $L$ by resolving {\em all} singular points via 3-ways $R_\pm$ and $R_0$. Indeed $\mathcal{R}(L)$ is indexed by $\mathcal{I}=\{f:\P(L)\to\{0,+,-\}\}$.

	\subsection{Stabilizations and classical invariants}
	
	For $L\in\LSK$ and a nonsingular point $p\in L\setminus\P(L)$, the {\em positive and negative stabilizations $S_\pm(L,p)$ of $L$ at $p$} are Legendrian singular links $S_\pm(L,p)$
	defined by the diagram replacement in the front projection as Figure~\ref{fig:stabilizations}. Note that $\|S_{\pm}(L,p)\|=\|L\|$ by definition.
	
	\begin{figure}[ht]
		\centering
		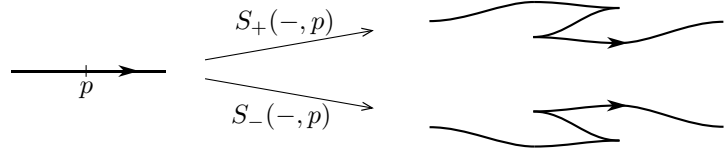
		\caption{Positive and negative stabilizations}
		\label{fig:stabilizations}
	\end{figure}
	
	\begin{lem}\cite{FT}\label{lem:mountain}
		For a nonsingular link $K\in\K$, two Legendrian links $L_1,L_2\in\mathcal{L}(K)$ are equivalent up to positive and negative stabilizations.
	\end{lem}
	On the contrary, it is not true that any two Legendrian singular links sharing the same singular link type can be connected by a sequence of positive and negative stabilizations in general. The corresponding result for $\LSK$ will be given in Proposition~\ref{prop:singconn}.
	
	For $L\in\LSK$, there are two {\em classical invariants}, which are generalizations of those in $\LK$ (see \cite{E}), and can be used to separate $\LSK$ as follows.
	For convenience sake, we label the components of $L$ as $L^1,\dots, L^n$.

	The {\em total Thurston-Bennequin number} $tb(L)$ measures the twisting of the contact structure along $L$, such that
	$tb(L)$ is a linking number $lk(L,L^+)$ with the positive push-off $L^+$, and therefore is invariant under contact isotopy.
	Indeed, each singular point of $L$ contributes $1$ or $-1$ to $tb(L)$ according to the orientation.
	Practically, it can be computed from $\pi_F(L)$ as
	\begin{align*}
		tb(L) =&   \#\left\{
		\vcenter{\hbox{\includegraphics[scale=1]{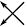}}},
		\vcenter{\hbox{\includegraphics[scale=1]{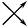}}},
		\vcenter{\hbox{\includegraphics[scale=1]{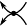}}}, 
		\vcenter{\hbox{\includegraphics[scale=1]{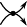}}}\right\}
		- \#\left\{         
		\vcenter{\hbox{\includegraphics[scale=1]{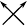}}},
		\vcenter{\hbox{\includegraphics[scale=1]{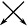}}},
		\vcenter{\hbox{\includegraphics[scale=1]{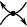}}}, 
		\vcenter{\hbox{\includegraphics[scale=1]{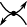}}},
		\vcenter{\hbox{\includegraphics[scale=1]{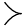}}}\right\}.
	\end{align*}
	
	Moreover, one can consider the Thurston-Bennequin number for each component $L^i$ as follows.
	$$\mathbf{tb}(L)=(tb(L^1),\dots,tb(L^n))\in\Z^n.$$
	It is easy to check that
	$$
	tb(L)=\sum_i tb(L^i) + \sum_{i<j} (lk(L^i, L^{j+}) + lk(L^i, L^{j-})).
	$$
	
	Notice that if $L$ is nonsingular, then both $lk(L^i, L^{j+})$ and $lk(L^i, L^{j-})$ are equal to $lk(L^i, L^j)$. However, these two linking numbers are different in general if $L$ is singular.
	
	Now we fix the trivialization of the contact structure $(\R^3,\xi_0)$ given by the Lagrangian projection.
	Then the {\em componentwise rotation number} $\mathbf{r}(L)=(r(L^1),\dots,r(L^n))\in\Z^n$ is also defined as the $n$-tuple of winding numbers $r(L^i)$ of tangent vectors of $L^i$ in the contact plane.
	In the front projection
	\begin{align*}
		r(L^i)=\frac{1}{2}\left(\#\{
		\vcenter{\hbox{\includegraphics{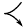}}},
		\vcenter{\hbox{\includegraphics{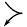}}}
		\} - \#\{
		\vcenter{\hbox{\includegraphics{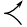}}},
		\vcenter{\hbox{\includegraphics{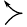}}}
		\}\right).
	\end{align*}
	
	We also define the {\em total} rotation number $r(L)$ by the sum of $r(L^i)$'s.
	Then it is easy to check that
	$$tb(S_\pm(L,p))=tb(L)-1,\quad r(S_\pm(L,p))=r(L)\pm 1,$$
	and for $\eta\in\{0,+,-\}$,
	$$
	tb(R_\eta(L,p))=tb(L)+\eta\cdot 1,\quad r(R_\eta(L,p))=r(L).
	$$
	
	\section{A hierarchy of invariants}
	
	\subsection{Legendrian simplicity}
	Recall that a nonsingular knot $K\in\K$ is {\em Legendrian simple} if $\mathcal{L}(K)$ are classified by $tb$ and $r$, and
	there are several knot types which are Legendrian simple. For example, the unknot, torus knots, the figure-8 knot $4_1$ are Legendrian simple \cite{EF, EH1}.
	For both singular and nonsingular links, it is natural to consider $\mathbf{tb}$ and $\mathbf{r}$, which are finer than $tb$ and $r$.
	Moreover, we can extend the notion of Legendrian simplicity as follows.
	\begin{defn}
		Let $I$ be a set of invariants of $\LSK$, and $K\in\SK$. We say that $K$ is {\em $I$-simple} if Legendrian singular links in $\mathcal{L}(K)$ are classified by the invariants in $I$.
	\end{defn}
	Then by definition, any Legendrian simple knot is $\{tb, r\}$-simple, and any split link with Legendrian simple components is not $\{tb,r\}$ but $\{\mathbf{tb},\mathbf{r}\}$-simple.
	
	The stabilizations $S_{\pm}$, defined in the previous section, obviously depend on where the kinks will be attached. However in $\LSK$, neither $\mathbf{tb}$ nor $\mathbf{r}$ have sufficient information to know that. A typical example, which is $\{\mathbf{tb},\mathbf{r}\}$-nonsimple,
	is the simplest singular knot $K_0$ of degree 1 with the singular point $\mathbf{0}$.
	Let $L_0\in\mathcal{L}(K_0)$ be the simplest Legendrian singular knot as follows.
	$$
	L_0=\vcenter{\hbox{\scriptsize
\begingroup%
  \makeatletter%
  \providecommand\color[2][]{%
    \errmessage{(Inkscape) Color is used for the text in Inkscape, but the package 'color.sty' is not loaded}%
    \renewcommand\color[2][]{}%
  }%
  \providecommand\transparent[1]{%
    \errmessage{(Inkscape) Transparency is used (non-zero) for the text in Inkscape, but the package 'transparent.sty' is not loaded}%
    \renewcommand\transparent[1]{}%
  }%
  \providecommand\rotatebox[2]{#2}%
  \ifx\svgwidth\undefined%
    \setlength{\unitlength}{80bp}%
    \ifx\svgscale\undefined%
      \relax%
    \else%
      \setlength{\unitlength}{\unitlength * \real{\svgscale}}%
    \fi%
  \else%
    \setlength{\unitlength}{\svgwidth}%
  \fi%
  \global\let\svgwidth\undefined%
  \global\let\svgscale\undefined%
  \makeatother%
  \begin{picture}(1,0.327625)%
    \put(0,0){\includegraphics[width=\unitlength]{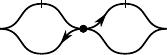}}%
    \put(0.2,0.205){\color[rgb]{0,0,0}\makebox(0,0)[lb]{\smash{$p$}}}%
    \put(0.7,0.205){\color[rgb]{0,0,0}\makebox(0,0)[lb]{\smash{$q$}}}%
    \put(0.45,0.035){\color[rgb]{0,0,0}\makebox(0,0)[lb]{\smash{$\mathbf{0}$}}}%
  \end{picture}%
\endgroup%
}}\in\mathcal{L}(K_0),\quad K_0=\vcenter{\hbox{\scriptsize
\begingroup%
  \makeatletter%
  \providecommand\color[2][]{%
    \errmessage{(Inkscape) Color is used for the text in Inkscape, but the package 'color.sty' is not loaded}%
    \renewcommand\color[2][]{}%
  }%
  \providecommand\transparent[1]{%
    \errmessage{(Inkscape) Transparency is used (non-zero) for the text in Inkscape, but the package 'transparent.sty' is not loaded}%
    \renewcommand\transparent[1]{}%
  }%
  \providecommand\rotatebox[2]{#2}%
  \ifx\svgwidth\undefined%
    \setlength{\unitlength}{40.8125bp}%
    \ifx\svgscale\undefined%
      \relax%
    \else%
      \setlength{\unitlength}{\unitlength * \real{\svgscale}}%
    \fi%
  \else%
    \setlength{\unitlength}{\svgwidth}%
  \fi%
  \global\let\svgwidth\undefined%
  \global\let\svgscale\undefined%
  \makeatother%
  \begin{picture}(1,0.41164055)%
    \put(0,0){\includegraphics[width=\unitlength]{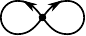}}%
    \put(0.58740219,0.1397427){\color[rgb]{0,0,0}\makebox(0,0)[lb]{\smash{$\mathbf{0}$}}}%
  \end{picture}%
\endgroup%
}}.
	$$
	Then two stabilizations at the marked points $p$ and $q$ in $L_0$ shown in Figure~\ref{fig:vertex_passing} are distinguished by $\{\mathbf{tb},\mathbf{r}\}$ after $0$-resolution, but never distinguished as they are. Note that this phenomena does not occur in $\LK$.
	
	\begin{figure}[ht]
		$$
		S_+^2(L_0,p)=\vcenter{\hbox{\includegraphics{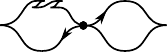}}}\neq
		\vcenter{\hbox{\includegraphics{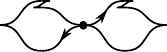}}}=S_+(S_+(L_0,p),q)
		$$
		\caption{Two different positive stabilizations of $L_0$ sharing $\{\mathbf{tb}, \mathbf{r}\}$}
		\label{fig:vertex_passing}
	\end{figure}
	
	Therefore it is natural to consider $\mathbf{tb}$ and $\mathbf{r}$ for all possible resolutions
	$\mathcal{R}(L)$, that is, $\{\mathbf{tb}(\mathcal{R}),\mathbf{r}(\mathcal{R})\}$.
	If the link types $\left\|\mathcal{R}(L)\right\|$ of resolutions of $L$ are nonsimple, however, 
	$\{\mathbf{tb}(\mathcal{R}),\mathbf{r}(\mathcal{R})\}$ can not capture the whole Legendrian information of $L$. So one could consider the set of Legendrian link types $\mathcal{R}(L)$ as invariants.
	Note that $\mathcal{R}$ is the strongest among all invariants mentioned above, since all nonsingular links are tautologically $\{\mathcal{R}\}$-simple.
	
	At first glance, all Legendrian link types in $\mathcal{R}(L)$ together with the singular link type $\|L\|$ seem to recover $L$ itself, but this is not true.
	For example, we will prove later that $K_0$ is $\{\mathcal{R}\}$-nonsimple.
	
	\subsection{Orders, markings, and flips}
	We discuss the distinctive properties of the singular points of Legendrian singular links.
	The {\em standard sphere} $S_{std}\subset \R^3$ is defined by
	$$
	S_{std}=\{(r,\theta,z)\,|\,r^4+4z^2=1\}
	$$
	in cylindrical coordinates.\footnote{The reason why we use this standard sphere will be explained in Appendix \ref{sec:projection}.} We denote by $B_{std}$ the inside of $S_{std}$, and call it the {\em standard 3-ball}.
	
	\begin{lem}\label{lem:Darboux}
		Let $L\in\LSK$ and $p\in \P(L)$.
		There exists a neighborhood $B_p\subset (S^3,\xi_{std})$ of $p$ and a contactomorphism $\phi_p$ between pairs of contact $3$-balls with co-orientation and oriented arcs such that
		$$
		\phi_p:(B_p,B_p\cap L)\to (B_{std},I_x\cup I_y)
		$$
		where $I_x=B_{std}\cap(x{\text-axis})$, and $I_y=B_{std}\cap(y{\text-axis})$.
	\end{lem}
	
	\begin{proof}
		By the Darboux theorem \cite[Theorem 2.5.1]{Ge}, there exists a neighborhood $U$ of $p$ 
		and a contactomorphism $\phi_0$ between $(U,U\cap L)$ and $(V,V\cap\phi_0(L))$ such that $\phi_0(p)=\mathbf{0}\in\R^3$ is a singular point. 
		We may assume that $V\cap\phi_0(L)$ is connected by choosing a small $U$. 
		
		We now parametrize $V\cap\phi_0(L)$ into two curves $\gamma_1(t)=(r_1(t),\theta_1(t),z_1(t))$, and $\gamma_2(t)=(r_2(t),\theta_2(t),z_2(t)),\ -\epsilon\leq t\leq\epsilon$ which match the orientation of $\phi_0(L)$ with the following conditions:
		\begin{enumerate}
			\item $\gamma_1(0)=\gamma_2(0)=0\in\R^3$;
			\item $r_1'(t)>0$, $r_2'(t)>0$ for $t\in(-\epsilon,\epsilon)$;
			\item $\theta_1(0)=0$, $\theta_2(0)\in(0,\pi)$;
			\item $|\theta_1(t)|<\frac{\delta}{3}$, $|\theta_2(t)-\theta_2(0)|<\frac{\delta}{3}$ where $\delta=\theta_2(0)-\theta_1(0)$.
		\end{enumerate}
		The conditions (2) and (4) are guaranteed by taking a sufficiently small neighborhood $U$ of $p$ and condition (3) is possible by the rotational symmetry of the contact structure $\xi_{rot}$. Note that conditions (2) and (3) determine the choice of $\gamma_1$ and $\gamma_2$.
		There is a Legendrian singular isotopy $h_s$, $s\in[0,1]$ which satisfies
		\begin{align*}
			h_s|_{\gamma_1(t)}&=(r_1(t),(1-s)\theta_1(t),(1-s)z_1(t)),\\
			h_s|_{\gamma_2(t)}&=(r_2(t),(1-s)\theta_2(t)+s\pi/2,(1-s)z_2(t)).
		\end{align*}
		and hence sends $\gamma_1(t)$ to $(r_1(t),0,0)$ and $\gamma_2(t)$ to $(r_2(t),\frac{\pi}{2},0)$ simultaneously. Let $\phi_1$ be a contact isotopy of $(\R^3,\xi_{rot})$ which realizes $h_t$. Then we may assume that $B_{std}\subset \phi_1(V)$.
		Consequently $B_p=\phi_0^{-1}\circ\phi_1^{-1}(B_{std})$ and $\phi_p=\phi_1\circ\phi_0|_{B_p}$ satisfies the desired condition.
	\end{proof}
	
	We call $B_p$ a {\em standard neighborhood} of $p$ and identify it with $B_{std}$ via $\phi_p$. Then $S_{std}\cap L$ consists of $\{\mathbf{0}_x,\mathbf{0}_y,-\mathbf{0}_x,-\mathbf{0}_y\}$, where $\mathbf{0}_x$ and $\mathbf{0}_y$ are the unit vectors along $x$ and $y$-axes.
	We simply denote $\phi_p^{-1}(\pm \mathbf{0}_x), \phi_p^{-1}(\pm \mathbf{0}_y)$ by ${\pm}p_x, \pm p_y$, which we collectively call the {\em nearby points} at $p$.
	Notice that the nearby points are well-defined up to reparametrization of $L$, which can be regarded as isotopy on the domain $nS^1$ of $L$ and thus safely ignored.
	
	\begin{defn}\label{def:order}
		Let $L\in\LSK$ and $B_p$ be a standard neighborhood of $p\in \P(L)$.
		An {\em order $\sigma(L,p)$} of $L$ at $p$ is a quadruple of nearby points of $p$ given by 
		$$
		\sigma(L,p)=(p_x,p_y,-p_x,-p_y).
		$$
	\end{defn}
	
	\begin{figure}[ht]
		$$
		\begin{array}{c|c|c|c|c}
		\text{Front}&\;\;
		\vcenter{\hbox{\scriptsize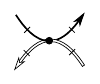}}\;\; &\;\;
		\vcenter{\hbox{\scriptsize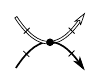}}\;\; &\;\;
		\vcenter{\hbox{\scriptsize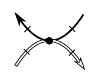}}\;\; &\;\;
		\vcenter{\hbox{\scriptsize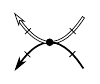}} \\
		
		\hline
		\text{Lagrangian}&\;\;
		\vcenter{\hbox{\scriptsize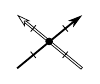}}\;\;&\;\;
		\vcenter{\hbox{\scriptsize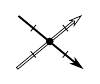}}\;\;&\;\;
		\vcenter{\hbox{\scriptsize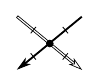}}\;\;&\;\;
		\vcenter{\hbox{\scriptsize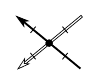}}
		\end{array}
		$$
		\caption{Projections near a singular point p} 
		\label{fig:nearbypoints}
	\end{figure}
	
	Since the contact structure $\xi_{std}$ is co-oriented, any co-orientation preserving contactomorphism $\phi:(S^3,L_1)\to(S^3,L_2)$ gives bijections between not only singular points but also nearby points up to contact isotopy. That is, 
	$$
	\phi(\pm p_x)=\pm q_x,\quad
	\phi(\pm p_y)=\pm q_y,
	$$
	where $p\in\P(L_1)$ and $q=\phi(p)\in\P(L_2)$.
	
	During the contact isotopy, the local shape `$\vcenter{\hbox{\includegraphics{char_L_singularcrossing_bicolor_dot_rev.pdf}}}$' at each singular point in the Lagrangian projection can be translated and rotated, but never flipped as `$\vcenter{\hbox{\includegraphics{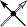}}}$' since our contact structure $(\R^3,\xi_0)$ is tight.
	Hence there is an obvious correspondence between singular points, and {\em ordered} nearby points as well. This implies the equivariance of $\sigma$ as above.
	
	In general, we can define the equivariant order for a singular point of a Legendrian singular link in any co-oriented contact 3-manifold $(M,\xi)$.
	
	Notice that $\sigma(L,p)$ is well-defined only for $L\in\LSK$ because there is no constraint on the tangent plane at $p\in\P(L)$ in $\SK$ and so it can be flipped freely.
	Consequently, the order $\sigma$ is a property exclusive to $\LSK$.
	
	We extend the concept of the order at singular points to arcs in Legendrian singular links.
	
	\begin{defn}\label{def:arc}
		Let $L\in\LSK$ and $\gamma:(I,\partial I)\to(L,\P(L))$ be an oriented arc which is piecewise smooth and injective except at $\P(L)$.
		A {\em marking $m(\gamma)$ of $\gamma$ on $L$} is a sequence of nearby points in $L$ that $\gamma$ meets.
	\end{defn}
	Note that the marking $m$ itself is also equivariant under co-orientation preserving contactomorphism $\phi$ as follows:
	$$
	\phi(m(\gamma))=m(\phi(\gamma)).
	$$
	Hence for any invariant $f$ on $\LSK$, we can consider the {\em enhanced invariant $f^m$ with marking $m$}, which may use the information from the marking $m$.
	For example, the results of enhanced full resolutions are Legendrian links with labels on each component.
	
	In general, the marking gives an obstruction for the given arc to be the same as another arc via contact isotopy, which will be discussed in \S\ref{sec:obs}.
	
	\begin{defn}
		Let $L\in\LSK$ and $p\in\P(L)$. 
		A flip move $Fl(L,p)$ is a diagram replacement of the front projection depicted as in Figure~\ref{fig:flip}.\footnote{This move has been discussed before. Indeed the flip $Fl$ is the {\em Legendrian horizontal flype} in \cite{NT}.}
	\end{defn}
	
	\begin{figure}[ht]
		$$
		\xymatrix{
			\vcenter{\hbox{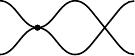}}\ar@{<->}[rr]^{Fl(L,p)}&&
			\vcenter{\hbox{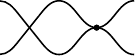}}
		}
		$$
		\caption{$Fl(L,p)$ of flipping at $p$}
		\label{fig:flip}
	\end{figure}
	
	Note that not all diagrams have a local picture as depicted in Figure \ref{fig:flip} so a flip is not always applicable.
	Moreover, it preserves the singular link type and resolutions and commutes with $S_\pm$. 
	That is,
	\begin{equation}\label{eq:flipresol}
	\|Fl\| = \|\cdot\|,\quad\mathcal{R}(Fl) = \mathcal{R},\quad
	Fl(S_\pm) = S_\pm(Fl),
	\end{equation}
	and see Figure~\ref{fig:resflip} for an example.
	
	\begin{figure}[ht]
		$$
		\xymatrix{
			\vcenter{\hbox{\includegraphics[scale=0.8]{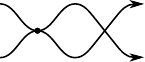}}}\ar@{<->}[rr]^{Fl(L,p)}\ar[dd]_-{R_+}\ar[ddr]_(0.45){R_-}\ar[ddrr]^(0.35){R_0}&&
			\vcenter{\hbox{\includegraphics[scale=0.8]{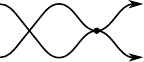}}}\ar[ddll]_(0.35){R_+}\ar[ddl]^(0.45){R_-}\ar[dd]^-{R_0}\\\\
			\vcenter{\hbox{\includegraphics[scale=0.8]{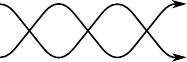}}}& 
			\vcenter{\hbox{\includegraphics[scale=0.8]{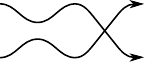}}}=
			\vcenter{\hbox{\includegraphics[scale=0.8]{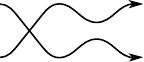}}}&
			\vcenter{\hbox{\includegraphics[scale=0.8]{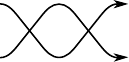}}}
		}
		$$
		\caption{Resolutions and flip moves}
		\label{fig:resflip}
	\end{figure}
	
	Another simple but important observation is that $\sigma(L,p)$ is not equivariant under $Fl$ in general, so it may not be realized by a Legendrian isotopy. Therefore the flip can not be replaced with $\pm$-(de)stabilizations since $S_\pm$ preserves $\sigma$, and we have a singular version of Lemma~\ref{lem:mountain}.

	\begin{prop}\label{prop:singconn}
		Let $K\in\SK$. Then any two Legendrian singular links $L_1$, $L_2$ in $\mathcal{L}(K)$ can be connected by a sequence of $S_\pm$ and $Fl$.
		Indeed, at most one flip for each singular point is necessary.
	\end{prop}
	\begin{proof}
		Since $L_1$ and $L_2$ have the same topological type, there is a smooth isotopy $\phi_t$ between $L_1$ and $L_2$. Consider 1-parameter family $\pi_F(\phi_t(L_1))$ of diagrams as before.
		Then by definition, both ends are {\em regular} but being regular fails for {\em almost all $t$} because the front projection at singular point depicted in Figure~\ref{fig:singularpoints} is never generic in the smooth setting.
		However, we avoid this anomalous situation by relaxing the definition of regularity of the front projection as follows.
		
		Let $v_1, v_2$ be two tangential vectors at $p\in\P(L_1)$. Then the plane generated by $v_i$'s may rotate during the isotopy $\phi_t$. 
		For the notational convenience, we introduce a function $\gamma_{p}:[0,1]\to S^2$ defined by
		$$
		\gamma_{p}(t)=\frac{(\phi_t)_*(v_1\times v_2)}{\|(\phi_t)_*(v_1\times v_2)\|}.
		$$
		We say that $\pi_F(\phi_t(L_1))$ is {\em almost regular near $p\in\P(L)$ at $t$} if $\gamma_{p}(t)$ does not lie on the equator $S^1_{xy}$.
		Then since $S^1_{xy}$ is closed in $S^2$, almost regularity is an open condition and therefore there are only finitely many exceptions for almost regularity near $p$.
		Moreover, since $S^2$ is simply connected, we may perturb $\phi_t$ so that $\gamma_{p}(t)$ intersects $S^1_{xy}$ at most once at $t(p)\in(0,1)$ for each $p\in\P(L)$.
		
		Let $\{t_1,\dots, t_k\}\subset I$ be the finite subset such that $\pi_F(\phi_t(L_1))$ is not almost regular near $p_i$ at $t_i$.
		Then for each interval $(t_{i-1},t_i)$ we can find a diagram $D_i$ by projecting tangential vectors $(\phi_t)_*(v_i)$ to the contact plane $\xi_{std}$ at $\phi_t(p_i)$ so that $D_i$ is regular near $\phi_t(p_i)$. Moreover, $D_{i+1}$ is obtained from $D_i$ by performing one flip move at $\phi_t(p_i)$ possibly with $\pm$-(de)stabilizations at the nearby points.
		
		Recall that all other kinds of failures of regularity correspond to $S_\pm$ together with Reidemeister moves depicted in Figure~\ref{fig:frontmove}.
	\end{proof}
	
	\begin{rmk}
		The following move which looks like a vertical flip preserves a Legendrian singular link type and can be obtained by applying the Reidemeister move {\rm (VI)} twice.
		\begin{figure}[ht]
			$$
			\xymatrix{
				\vcenter{\hbox{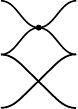}}\ar@{<->}[rr]^{({\rm VI})\circ({\rm VI})}&&
				\vcenter{\hbox{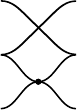}}
			}
			$$
			\caption{A vertical flip}
			\label{fig:vflip}
		\end{figure}
		
	\end{rmk}

	\subsection{Obstruction from the marking}\label{sec:obs}
	Let $L_1, L_2\in\LSK$, and $\gamma$ be an arc in $L_1$ defined as in Definition \ref{def:arc}.
	Then the marking $m(\gamma)$ corresponds to a smooth arc of some full resolution $\mathcal{R}$ of $L_1$. Especially, if $\gamma$ is a loop, $m(\gamma)$ represents a link component of $\mathcal{R}(L_1)$, denote it by $\mathcal{R}(\gamma)$.
	
	Now suppose that there is a contact isotopy $\phi$ between $L_1$ and $L_2$.
	Since the marking is equivariant under $\phi$, the link component $\mathcal{R}(\gamma)$ of $\mathcal{R}(L_1)$ should map to the link component $\mathcal{R}(\phi(\gamma))$ of $\mathcal{R}(L_2)$.
	In other words, if there is no contact isotopy between $\mathcal{R}(L_1)$ and $\mathcal{R}(L_2)$ sending 
	$\mathcal{R}(\gamma)$ to $\mathcal{R}(\phi(\gamma))$ then $L_1$ is different from $L_2$ in $\LSK$.
	
	Hence the obstructions obtained in this way are related with more intrinsic structures of Legendrian or smooth link types, such as, {\em topological non-switchability} (Example~1), {\em Legendrian non-switchability} (Example~2), and {\em Legendrian non-invertibility} (Example~3).
	
	\begin{ex}[Topological non-switchability]
		Let $L\in\LK$ be a Legendrian knot with $\|L\|$ different from the unknot, and $L_{1}, L_{2}$ be Legendrian singular knots of degree 1 with $p\in\P(L_{1})$, $q\in\P(L_{2})$ as in Figure~\ref{fig:pmsingstab}.
		One can directly check that $\|L_1\|=\|L_2\|$ in $\SK$ and furthermore $\mathcal{R}(L_1)=\mathcal{R}(L_2)$ in $\LK$.
		
		\begin{figure}[ht]
			$$
			L=\vcenter{\hbox{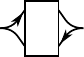}}, \quad
			L_{1}=\vcenter{\hbox{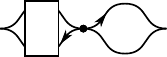}},\quad
			L_{2}=\vcenter{\hbox{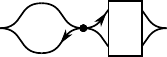}}
			$$
			$$
			R_+(L_{1},p)=\vcenter{\hbox{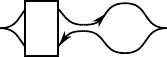}}=
			\vcenter{\hbox{\input{F_L.pdf_tex}}}=\vcenter{\hbox{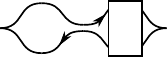}}=R_+(L_{2},q)
			$$
			$$
			R_-(L_{1},p)=\vcenter{\hbox{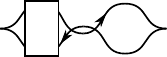}}=S_+S_-(L)=
			\vcenter{\hbox{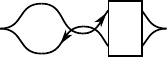}}=R_-(L_{2},q)
			$$
			$$
			R_0(L_{1},p)=\vcenter{\hbox{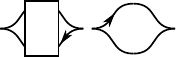}}\ =L\coprod L_\bigcirc=\
			\vcenter{\hbox{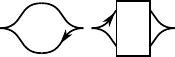}}=R_0(L_{2},q)
			$$
			\caption{Positive and negative singular stabilizations}
			\label{fig:pmsingstab}
		\end{figure}
		
		Let $\gamma_1, \gamma_2$  be arcs in $L_{1}$ starting and ending at $p$ which satisfy
		$$
		m(\gamma_1)=(p_x,-p_y),\quad m(\gamma_2)=(p_y,-p_x).
		$$
		Then $\mathcal{R}(\gamma_1)$ and $\mathcal{R}(\gamma_2)$ correspond to the unknot and $L$ in $R_0(L_{1},p)$, respectively.
		But in $R_0(L_2,q)$ the corresponding components are switched.
		Therefore $L_1$ and $L_2$ are different Legendrian singular knots in $\LSK$, and never connected by stabilizations because both stabilizations $S_\pm$ commute with the $0$-resolution $R_0$.
		Hence one flip is necessary to connect them in $\mathcal{L}(\|L_1\|)$ by Proposition~\ref{prop:singconn}.
	\end{ex}
	
	\begin{ex}[Legendrian non-switchability]
		Let us provide another pair of Legendrian singular knots of degree 1, $(L_{a},p)$, $(L_{b},q)$ depicted in Figure~\ref{fig:6_3,8_16}.
		It is easily checked that they are different  by exactly one flip move and   by (\ref{eq:flipresol}) we have 
		\begin{align*}
			\|L_{a}\|&=\|L_{b}\|\in\SK,\\ 
			R_+(L_{a},p)&=R_+(L_{b},q)\in\mathcal{L}(8_2),\\
			R_-(L_{a},p)&=R_-(L_{b},q)\in\mathcal{L}(6_2),\\
			R_0(L_{a},p)&=R_0(L_{b},q)\in\mathcal{L}(L7a6),
		\end{align*}
		where $L7a6$ is a link in the Thistlethwaite link table. 
		
		Suppose $L_a=L_b$ via a contact isotopy $\phi_t$ in $\LSK$.
		Let $\gamma_1, \gamma_2$ be arcs in $L_a$ starting and ending at $p$ satisfying that 
		$$
		m(\gamma_1)=(p_y,-p_x)\quad\text{ and }\quad
		m(\gamma_2)=(p_x,-p_y).
		$$
		
		Figure \ref{fig:R_0(S_12)} shows two components of both $R_0(L_a,p)$ and $R_0(L_b,q)$, determined by the markings $m(\gamma_i)$ and $m(\phi_1(\gamma_i))$, respectively.
		Note that
		$$
		m(\phi(\gamma_1))=(q_y,-q_x)\quad\text{ and }\quad
		m(\phi(\gamma_2))=(q_x,-q_y).
		$$
		
		Hence, $\phi_1$ must switch the components as shown in Figure~\ref{fig:R_0(S_12)}, and the link type $\|R_0(L_a,p)\|=L7a6$ is topologically switchable.
		However $R_0(L_a,p)$ is not Legendrian switchable, i.e., there is no Legendrian isotopy interchanging its components, see \cite{Cho}. Therefore this contradiction implies that $L_a\neq L_b$ in $\LSK$.
		
		\begin{figure}[ht]
			\centering
			{\scriptsize
				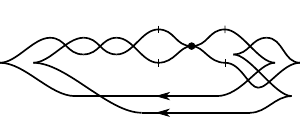\qquad
				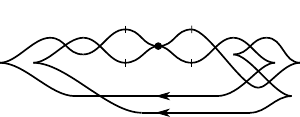
			}
			\caption{A pair of Legendrian singular knots $L_a$ and $L_b$}
			\label{fig:6_3,8_16}
		\end{figure}
		
		\begin{figure}[ht]
			\centering
			{\scriptsize
				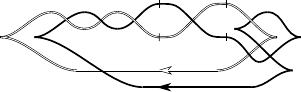\qquad
				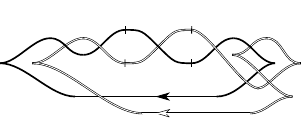
			}
			\caption{Two $0$-resolutions $R_0(L_{a})$ and $R_0(L_{b})$}
			\label{fig:R_0(S_12)}
		\end{figure}
	\end{ex}
	
	\begin{ex}[Legendrian non-invertibility]
		Since there is no canonical orientation of $R_\infty$, it seems less natural compared to the other three resolutions.
		However, by using the aid of marking, there is a way for assigning an orientation consistently as follows.
		For $L\in\LSK$ and $p\in\P(L)$, the $\infty$-resolution $R_\infty(L,p)$ at $p$ is a modification so that $p_x$ and $p_y$ (or $-p_x$ and $-p_y$) are joined by an arc.
		Hence by the equivariance of marking, we may assign an orientation near $p$ as {\em from $p_x$ to $p_y$}, or the opposite way.
		It is easy to check that this assignment defines an orientation on $R_\infty(L,p)$ no matter how the arcs passing through $p$ are joined in $L$ globally.
		This is the enhancement $R_\infty^m$ of the $\infty$-resolution $R_\infty$.
		
		There exists a pair of examples which can be distinguished by $R_\infty^m$ but not by the classical invariants and $R_\eta$, $\eta\in\{+,-,0\}$ as follows. 
		Let $(L_c,p)$ and $(L_{d},q)$ be Legendrian singular knots of degree one as depicted in Figure~\ref{fig:6_1,8_1}. 
		\begin{figure}[ht]
			\centering
			{\scriptsize
				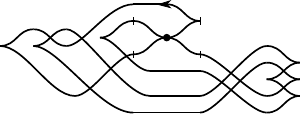\qquad
				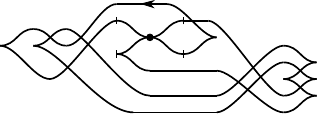
			}
			\caption{A pair of Legendrian singular knots $L_c$ and $L_d$}
			\label{fig:6_1,8_1}
		\end{figure}
		
		Since $L_d$ can be obtained by one negative flip move as before, one can check that
		\begin{align*}
			\|L_c\|&=\|L_d\|\in\SK,\\
			R_+(L_{c},p)&=R_+(L_{d},q)\in\mathcal{L}(8_1),\\
			R_-(L_{c},p)&=R_-(L_{d},q)\in\mathcal{L}(6_1),\\
			R_0(L_{c},p)&=R_0(L_{d},q)\in\mathcal{L}(K_H).
		\end{align*}
		Here $K_H$ is the Hopf link having the linking number $-1$.
		Moreover, $R_0^m(L_{c},p)$ and $R_0^m(L_{d},q)$ are {\em labelled} Legendrian Hopf links which look like
		$$
		R_0^m(L_{c},p)=\vcenter{\hbox{\includegraphics{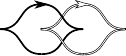}}},\quad
		R_0^m(L_{d},q)=\vcenter{\hbox{\includegraphics{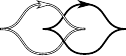}}}.
		$$
		Therefore, they are same as Legendrian links, and the enhanced $0$-resolution $R_0^m$ is not useful for this pair.
		
		As mentioned above, the enhanced $\infty$-resolutions $R^m_\infty(L_c,p)$ and $R^m_\infty(L_d,q)$ can be considered as oriented Legendrian knots, whose orientations are given by arcs from $p_x$ to $p_y$ and $q_x$ to $q_y$, respectively.
		More precisely, we have
		$$
		R^m_\infty(L_c,p) = S_-( L(\mu(7_2))),\quad
		R^m_\infty(L_d,q) = S_-( L(-\mu(7_2))),
		$$
		where $\mu(7_2)$ is a topological mirror of $7_2$ knot, and $L(\mu(7_2))$ looks like as follows.
		$$
		\includegraphics{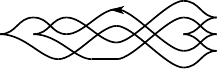}
		$$
		
		It is known that $L(\mu(7_2))$ is Legendrian non-invertible\footnote{Here a given Legendrian knot $L$ is {\em non-invertible} means that $L\neq -L$ as a Legendrian knot type and $L(\mu(7_2))$ is the simplest Legendrian non-invertible knot.}, that is, $L(\mu(7_2))$ and $-L(\mu(7_2))$ are not same in $\LK$.
		Moreover, their stabilizations are pairwise different as well \cite{Cho}, and so $R^m_\infty(L_c,p)\neq R^m_\infty(L_d,q)$. 
		
		Therefore $L_c\neq L_d$ in $\LSK$, and this means that the enhanced resolutions are strictly stronger than the Legendrian switchabilities of resolutions as obstructions.
	\end{ex}
	
	\begin{rmk}
		Topological non-invertibility can be used to produce another distinct pairs in $\LSK$.
		Then as Example~1, the resulting pairs are not connected by a sequence of stabilizations.
	\end{rmk}
	
	\begin{ex}\label{ex:final}
		In Example~1, 2 and 3, we heavily use the properties of link types, such as switchability and invertibility.
		But there still exist subtle phenomena which are not captured by any invariant defined above. 
		Let $L_e$ be the Legendrian singular knot described in Figure~\ref{fig:0,4_1} with $p\in\P(L_e)$.
		
		\begin{figure}[ht]
			\centering
			{\scriptsize
				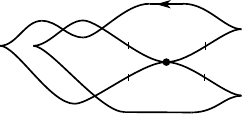\qquad
				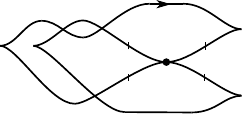
			}
			\caption{A pair of Legendrian singular knots $L_e$ and $-L_e$}
			\label{fig:0,4_1}
		\end{figure}
		
		One can readily check that $L_e$ and $-L_e$ share the all the invariants defined above.
		In order to distinguish them we need a certain preparation, a singular connected sum or a tangle replacement near the singular point.
		We will come back to this example when we are ready.
	\end{ex}

	\section{Singular connected sum and decomposition}
	\label{sec:singularconnectedsum}
	The main content of this section is to define a {\em singular connected sum of two Legendrian singular links} as a generalization of the connected sum of two Legendrian links. 
	
	\subsection{Singular connected sum}\label{subsec:singularconnectedsum}
	For a given $L\in \LSK$ and $p\in\P(L)$, we have the local standard neighborhood $B_p$ of $p$ with a contactomorphism $\phi_{L,p}$ as in Lemma \ref{lem:Darboux}.
	$$
	\phi_{L,p}:(B_p, B_p\cap L)\to(B_{std},I_x\cup I_y).
	$$
	
	Let $(L_1,p)$ and $(L_2,q)$ be pairs of Legendrian singular links and singular points, and let $B_p$ and $B_q$ be standard neighborhoods.
	We define $\phi:\partial B_p\to\partial B_q$ as the composition of three diffeomorphisms 
	\begin{align*}
		\phi:(\partial B_p, \sigma(L_1,p))&\stackrel{\phi_{L_1,p}}{\longrightarrow}(S_{std},\sigma(I_x\cup I_y,\mathbf{0}))\\
		&\stackrel{-_{xy}}{\longrightarrow}
		(S_{std},-\sigma(I_x\cup I_y,\mathbf{0}))\stackrel{\phi_{L_2,q}^{-1}}{\longrightarrow}(\partial B_q,-\sigma(L_2,q)),
	\end{align*}
	$$
	$$
	where $-_{xy}(x,y,z)=(-x,-y,z)$ is $\pi$-rotation along the $z$-axis\footnote{If a $0$, $\pi/2$ or $-\pi/2$ rotation is used instead, then it defines an {\em unoriented singular connected sum}. See \S\ref{sec:unori}.}.
	Hence $\phi$ maps nearby points of $p$ to those of $q$ as
	$$\phi:(p_x,p_y,-p_x,-p_y)\mapsto(-q_x,-q_y,q_x,q_y).$$
	
	Then the connected sum of two $(S^3,\xi_{std}$) can be defined by using the gluing map $\phi$. To give an orientation on the connected sum $S^3\# S^3$, it is necessary that 
	either $\phi_{L_{1,p}}$ or $\phi_{L_{2,q}}$ is orientation-reversing.
	Then $\phi$ is an orientation-reversing diffeomorphism. 
	Note that $\phi$ gives an orientation-reversing isomorphism on the oriented characteristic foliations $(\partial B_p)_{\xi_{std}}$ and $(\partial B_q)_{\xi_{std}}$.
	Then by Colin's gluing theorem \cite{Co} the resulting manifold is again $(S^3,\xi_{std})$.
	
	\begin{defn}
		The {\em singular connected sum $(L_1,p)\otimes(L_2,q)$} is the Legendrian singular link in $S^3$ defined by
		$$
		(L_1,p)\otimes(L_2,q)=(L_1\setminus (L_1\cap \mathring{B}_p))\coprod_\phi(L_2\setminus (L_2\cap \mathring{B}_q)).
		$$
	\end{defn}
	
	\begin{proof}[Proof of Theorem \ref{thm:singularconnectedsum}]
		Notice that the only possible ambiguities occur when we choose standard neighborhoods.
		If there are two standard neighborhoods, then we may assume that one contains the other.
		However, the complementary region is diffeomorphic to $S^2\times [0,1]$ whose contact structure is determined uniquely by the characteristic foliations at boundaries up to contact isotopy \cite[Theorem 4.9.4]{Ge}.
		Hence all standard neighborhoods are contact isotopic in $S^3$.
	\end{proof}
	
	When $p\in L_1$ and $q\in L_2$ are non-singular, the standard neighborhoods $B_p$ and $B_q$ can be identified with $(B_{std},I_x)$.
	Then the above gluing homeomorphism $\phi:\partial B_p\to \partial B_q$ recovers the usual connected sum $(L_1,p)\# (L_2,q)$ discussed in \cite{EH}.
	Roughly speaking, the singular connected sum looks like usual connected sums of two pairs of components simultaneously. 
	
	Recall that a Legendrian unknot $L_\bigcirc$ is the identity of the connected sum operation. The following plays a role of the identity under the singular connected sum.
	\begin{defn} \label{def:identity}
		Let $L_{\bigcirc\!\!\!\bigcirc}$ be a 2-component Legendrian singular link of degree $2$ defined by 
		$$
		L_{\bigcirc\!\!\!\bigcirc}=(x\text{-axis})\cup (y\text{-axis})\cup\{\infty\}\subset\R^3\cup\{\infty\} = S^3.
		$$
	\end{defn}
	
	Alternatively, since each axis represents the Legendrian unknot $L_\bigcirc$ in $S^3$, $L_{\bigcirc\!\!\!\bigcirc}$ is a union of 2 copies of $L_\bigcirc$ with 2 singular points $\{\mathbf{0},\infty\}$,
	as depicted in Figure~\ref{fig:tanglesumidentity}.
	Since $(S^3, L_{\bigcirc\!\!\!\bigcirc})$ is obtained by gluing two copies of $(B_{std}, I_x\cup I_y)$, it is the identity under the singular connected sum. Note that since $L_{\bigcirc\!\!\!\bigcirc}$ has rotational symmetry, it has a unique choice of orientation up to isotopy.
	
	\begin{figure}[ht]
		{\scriptsize
			$$
			\vcenter{\hbox{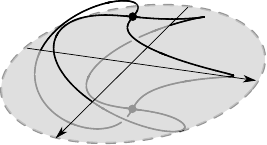}}\quad
			\vcenter{\hbox{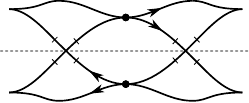}}\quad
			\vcenter{\hbox{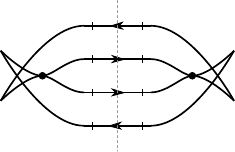}}
			$$
		}
		\caption{The Legendrian singular link $L_{\bigcirc\!\!\!\bigcirc}$ and its projections}
		\label{fig:tanglesumidentity}
	\end{figure}
	
	\subsection{Unoriented singular connected sum}
	\label{sec:unori}
	Let $(L_1,p), (L_2, q)$ be pairs of Legendrian singular links and singular points.
	Then the orders $\sigma(L_1,p)$ and $\sigma(L_2,q)$ are
	$$
	\sigma(L_1,p)=(p_x,p_y,-p_x,-p_y),\quad\sigma(L_2,q)=(q_x,q_y,-q_x,-q_y).
	$$
	
	Recall the gluing map $\phi$ defined by $\phi(\pm p_*)=\mp q_*$, which uses the $\pi$-rotation about $z$-axis and the only option resulting in a canonical orientation of $(L_1,p)\otimes(L_2,q)$.
	However, if we relax the condition about the induced orientation, there are 3 more options $\phi_0$, $\phi_+$ and $\phi_-$ to glue nearby points of $L_1$ and $L_2$, where $\phi_0$ uses 0-rotation and $\phi_\pm$ uses the $\pm\pi/2$-rotation about $z$-axis, respectively.
	In other words,
	\begin{align*}
		\phi_0(p_x,p_y,-p_x,-p_y)&=(q_x,q_y,-q_x,-q_y)\\
		\phi_+(p_x,p_y,-p_x,-p_y)&=(q_y,-q_x,-q_y,q_x)\\
		\phi_-(p_x,p_y,-p_x,-p_y)&=(-q_y,q_x,q_y,-q_x).
	\end{align*}
	
	Then for $\eta\in\{+,-,0\}$, we define the {\em $\eta$-unoriented singular connected sum $(L_1,p)\otimes_{\eta}(L_2,q)$} by using $\phi_{\eta}$ as follows.
	$$
	(L_1,p)\otimes_{\eta}(L_2,q)=(L_1\setminus (L_1\cap \mathring{B}_p))\coprod_{\phi_{\eta}}(L_2\setminus(L_2\cap \mathring{B}_q)).
	$$
	It is obvious that as unoriented Legendrian links
	\begin{align*}
		(L_1,p)\otimes_0(L_2,q)&=|(L_1,p)\otimes(-L_2,q)|=|(-L_2,q)\otimes(L_1,p)|\\
		&=|(L_2,q)\otimes(-L_1,p)|=(L_2,q)\otimes_0(L_1,p).
	\end{align*}
	where $|L|$ and $-L$ are obtained by forgetting and reversing orientations of $L$, respectively. Therefore $\otimes_0$ is commutative.
	However, neither $\otimes_+$ nor $\otimes_-$ is commutative. Instead, we have
	$$
	(L_1,p)\otimes_+(L_2,q)=(L_1,p)\otimes_-(-L_2,q).
	$$
	
	Indeed, when one of $L_i$'s is the same as its reverse, then both $\otimes_\pm$ are the same and commutative on the $L_i$'s. The $\infty$-resolution is a typical example.

	\subsection{Singular connected sum decomposition}
	Recall the standard sphere $S_{std}$ is defined by the equation $r^4+4z^2=1$. Then 
	its characteristic foliation $(S_{std})_{\xi_{rot}}$ given by $\alpha_{rot}$ looks as depicted in Figure~\ref{fig:standardball}.
	\begin{figure}[ht]
		\includegraphics{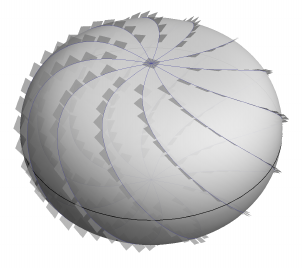}
		\caption{The standard sphere $S_{std}$ in $(\R^3,\xi_{rot})$}
		\label{fig:standardball}
	\end{figure}
	
	A {\em separating sphere $S$ of $L$} is a sphere in $S^3$ such that
	\begin{enumerate}
		\item there exists a contactomorphism $(\R^3,S)\to(\R^3,S_{std})$, and therefore an oriented characteristic foliation $S_{\xi_{std}}$ has exactly two elliptic singular points $e^+$ and $e^-$ (such a characteristic foliation is called {\em standard});
		\item $S$ intersects $L$ transversely at four points;
		\item each intersection in $S\cap L$ lies in a distinct leaf of $S_{\xi_{std}}$.
	\end{enumerate}
	
	Then there is a  projection map $\tau:S\setminus\{e^+,e^-\}\to S^1$ along the leaves which allows us to define the following.
	
	\begin{defn}\label{def:cyclicorder}
		Let $L\in\LSK$ and $S$ be a separating sphere of $L$.
		A {\em cyclic order $\sigma_{cyc}(L,S)$} is defined as an order on $L\cap S$ up to cyclic permutations
		induced by $\tau$.
		
		A contact isotopy $H_t$ is {\em order-preserving on $S$ with respect to $L$} if $H_t(S)$ is a separating sphere of $L$ for each $t$.
	\end{defn}
	
	Then we have the following lemma which is an analogue of Lemma~\ref{lem:Darboux}.
	
	\begin{lem}
		Let $L\in\LSK$ and $S$ be a separating sphere of $L$. Then there exists a neighborhood $N(S)\subset S^3$ of $S$ and a contactomorphism $\phi_S$ of pairs such that
		$$
		\phi_S:(N(S), N(S)\cap L)\to(\R^3\setminus\{\mathbf{0}\}, I_x\cup I_y\setminus\{\mathbf{0}\}).
		$$
		
		Moreover $\phi_S(S)$ and $S_{std}$ are order-preserving contact isotopic with respect to $\phi_S(L)$.
	\end{lem}
	\begin{proof}
		By definition, there is a neighborhood $N_0(S)\subset S^3$ and contactomorphism $\phi_0:N_0(S)\to\R^3\setminus\{\mathbf{0}\}$, where $\phi_0(S)=S_{std}$.
		Now consider $\phi_0(L)$ as a parametrized curve $(r(t),\theta(t),z(t))$ with respect to the cylindrical coordinate. Then we can perturb $L$ slightly to obtain $r'\neq 0$ on $S_{std}\cap L$.
		Therefore there is a small enough $\epsilon>0$ such that on $S_{std}\times(-\epsilon,\epsilon)\cap L\subset \R^3$,
		$$
		r'\neq 0,\quad
		-\pi<\theta<\pi,\quad
		\epsilon\cdot\max{|\theta'|}\ll1.
		$$
		
		We identify $S_{std}\times(-\epsilon,\epsilon)$ with $\R^3\setminus\{\mathbf{0}\}$ via $\phi$. Then the image $(\phi_1\circ\phi_0)(L)$ of $L$ is a union of four arcs which are strictly increasing in the radial directions, and by the same isotopies as in the proof of Lemma~\ref{lem:Darboux}, they can be isotoped to $I_x\cup I_y\setminus\{\mathbf{0}\}$ via $H_t$. Then by the choice of $\epsilon$, the order $\sigma(H_t((\phi_1\circ\phi_0)(L)),S_{std})$ is well-defined for all $t$, and we can let $\phi_2=H_1$.
		
		We define $N(S)$ as $\phi_0^{-1}(S_{std}\times(-\epsilon,\epsilon))$, and $\phi_S$ as $\phi_2\circ\phi_1\circ\phi_0$. Then $N(S)$ and $\phi_S$ are the desired neighborhood and contactomorphism.
	\end{proof}

	We consider a decomposition for $L\in\LSK$ which is an inverse of the singular connected sum.
	At first, suppose a separating sphere $S$ of $L$ is given.
	If we assign $+$ or $-$ for each point in $S\cap L$ according to the orientation of $L$, 
	then the sign of $\sigma_{cyc}(L,S)$ is either $(+,+,-,-)$ or $(+,-,+,-)$\footnote{This corresponds to the {\em unoriented singular connected sum} as before. See \S\ref{sec:unori}.} up to cyclic permutations.
	However, the latter case $(+,-,+,-)$ is not the configuration we want
	because a singular connected sum never gives this kind of order.
	
	We define a non-cyclic {\em order} $\sigma(L,S)$ by a representative of $\sigma_{cyc}(L,S)$ 
	whose signs realize $(+,+,-,-)$, when it is possible.
	Note that this sign configuration coincides with Definition~\ref{def:order}.
	
	Now we consider a separating sphere $S_0=\phi_S^{-1}(S_{std})$ given by the lemma above.
	Then $S_0$ bounds two 3-balls $B_1$ and $B_2$, which are contactomorphic to $B_{std}$ via $\phi_1$ and $\phi_2$.
	We may assume that $\phi_i$ and $\phi_S$ coincide on $B_i\cap N(S)$. Therefore $\phi_i(L\cap B_i)$ satisfies the conditions for the Legendrian tangle.
	
	We define two Legendrian singular links $L_i$ as closures of tangles $\phi_i(L\cap B_i)$, or equivalently,
	$$
	L_i=(L\cap B_i)\coprod_{-\partial\phi_S} (I_x\cup I_y)\subset B_i\coprod_{-\partial\phi_S} B_{std}=S^3
	$$
	where $-\partial\phi_S=(-_{xy}\circ\phi_S):S_0\to S_{std}$ is a composition of $\phi_S$ and $-_{xy}$, and it maps
	$\sigma(L,S_0)$ to $-\sigma(I_x\cup I_y,\mathbf{0})$.
	
	\begin{prop}
		Let $L, S$ and $L_i$'s be as above.
		Then $L=(L_1,\mathbf{0})\otimes(L_2,\mathbf{0})$.
	\end{prop}
	\begin{proof}
		This follows obviously from the well-definedness of the singular connected sum.
	\end{proof}
	
	We prove Theorem~\ref{thm:singularconnectedsumdecomposition}, the well-definedness of the singular connected sum decomposition up to order-preserving contact isotopy.
	\begin{proof}[Proof of Theorem~\ref{thm:singularconnectedsumdecomposition}]
		Let $S$ and $S'$ be separating spheres of $L$ which are order-preserving contact isotopic.
		We choose $S_0$ and $S_0'$ as before and obtain the singular connected summands $L_i$ and $L_i'$ by using $S_0$ and $S_0'$.
		Then it suffices to show that $L_i=L_i'$.
		There are two parts where the ambiguities can occur, but we may assume that $S=S_0$ and $S'=S_0'$. In other words, both $\phi_S(S)$ and $\phi_{S'}(S')$ are $S_{std}$.
		
		Let $H_t$ be the order-preserving contact isotopy between $S$ and $S'$, and let $B_i'$ be two 3-balls that $S'$ bounds.
		Then without loss of generality, we may assume that $S$ and $S'$ are disjoint and bound a subspace diffeomorphic to $S^2\times I$, by dividing the interval $I=[0,1]$ and by the convexity of $H_t(S)$ for all $t$, see \cite[Lemma~4.12.3~(ii)]{Ge}.
		
		Therefore there is a contact embedding $\iota:S^2\times I\to S^3$ such that $\iota_0(S^2)=S$ and $\iota_1(S^2)=S'$, and induce the isomorphisms between characteristic foliations.
		Hence $B_1\coprod_{\iota_0} (S^2\times I) =B_1'$ and $(S^2\times I)\coprod_{\iota_1} B_2' = B_2$.
		Moreover, it is obvious that $\iota^{-1}(L)$ is Legendrian in $S^2\times I$ and the order $\sigma(\iota^{-1}(L), S^2\times\{t\})$ is well-defined for all $t$.
		
		Then we consider the singular Legendrian link $L_{S,S'}$ defined by
		\begin{align*}
			L_{S,S'}=&(I_x\cup I_y)\coprod_{-\partial\phi_S\circ \iota_0} \iota^{-1}(L)\coprod_{-\partial\phi_{S'}\circ \iota_1} (I_x\cup I_y)\\
			\subset& B_{std}\coprod_{-\partial\phi_S\circ \iota_0} S^2\times I\coprod_{-\partial\phi_{S'}\circ \iota_1}B_{std}=S^3.
		\end{align*}
		
		\begin{lem}
			The singular Legendrian link $L_{S,S'}$ is the same as $L_{\bigcirc\!\!\!\bigcirc}$.
		\end{lem}
		\begin{proof}
			It suffices to show that $L_{S,S'}$ lies in a sphere whose characteristic foliations are standard.
			We construct such a sphere $S_{S,S'}$ as follows.
			
			Choose two standard discs $D_1, D_2$ containing $I_x\cup I_y$'s in the standard neighborhood of two singular points in $L_{S,S'}$. Then the characteristic foliation on each $D_i$ has exactly one singularity, which is elliptic.
			Since the order $\sigma(\iota^{-1}(L), S\times \{t\})$ is well-defined for each $t\in[0,1]$, there exists a circle $S^1_t\subset S\times\{t\}$ which is transverse to foliations on $S\times\{t\}$ and passes through four intersection points $\iota^{-1}(L)\cap (S^2\times\{t\})$.
			We may choose a family $S^1_t$ of circles as varying smoothly by $t$.
			
			Hence it defines an annulus $A$ whose characteristic foliations has no singularity by definition, and we obtain the desired sphere $S_{S,S'}$ by gluing two discs $D_1, D_2$.
		\end{proof}
		
		This lemma directly implies that 
		\begin{align*}
			L_1&=(L\cap B_1)\coprod_{-\partial\phi_S} (I_x\cup I_y)\\
			&=(L_1,\mathbf{0}) \otimes (L_{\bigcirc\!\!\!\bigcirc},\mathbf{0})=(L_1,\mathbf{0}) \otimes (L_{S,S'},\mathbf{0})\\
			&= (L\cap B_1)\coprod_{\iota_0} (S^2\times I)\coprod_{-\partial\phi_{S'}\circ \iota_1} (I_x\cup I_y)\\
			&= (L\cap B_1')\coprod_{-\partial\phi_{S'}\circ \iota_1} (I_x\cup I_y) = L_1'
		\end{align*}
		Similarly, we have $L_2=L_2'$ by the same argument, and Theorem~\ref{thm:singularconnectedsumdecomposition} is proved.
	\end{proof}
	
	Recall that for $K\in \SK$, the existence of an embedded sphere $S$ with $|K\cap S|=0$ or $2$
	ensures that $K$ can be decomposed into $K_i$'s via the disjoint union or the usual connected sum, respectively.
	For $L\in\LSK$, we can decompose $L$ further via a separating sphere $S$ with well-defined $\sigma(L,S)$. 
	
	\begin{rmk}
		One may ask whether similar notions of the singular connected sum and decomposition are possible in more general settings such as $\SK$ or $4$-valent graphs $\mathcal{V}_4$.
		
		In the case of $\SK$, the singular connected sum is not well-defined. Because of the lack of order there are two possibilities. The singular connected sum decomposition in $\SK$, however, has as many possibilities as the mapping classes of $S^2$ with $4$-marked points.
		
		On the contrary, the decomposition in $\mathcal{V}_4$ is well-defined as in \cite{M}, since the flexibility of vertices excludes ambiguities. A corresponding operation to the singular connected sum in $\mathcal{V}_4$ also has the same ambiguities as the mapping classes of $S^2$ with $4$-marked points.
	\end{rmk}
	\begin{table}[ht]
	\setlength{\tabcolsep}{1.4pc}
		\begin{tabular}{c|c|c|c}
		& $\SK$ & $\mathcal{V}_4$ & $\LSK$\\
		\hline
		 singular connected sum & 2 & $\infty$ & 1\\
		\hline
		 decomposition & $\infty$ & 1 & 1 \\
		\end{tabular}
		\vspace{1em}
		\caption{The numbers of ambiguities for the singular connected sum and the decomposition}
		\label{table:ambiguities}
	\end{table}
	
		\subsection{Units of singular connected sum}
	Now we consider unit elements of the singular connected sum, which is a left (or right) summand of the identity.
	
	\begin{defn}
		A pair $(L,p)$ with $L\in\LSK$ and $p\in\mathcal{P}(L)$ is a {\em unit} if 
		there exists $L'$ and $p'\in\mathcal{P}(L')$ such that
		$$
		(L,p)\otimes(L',p')=L_{\bigcirc\!\!\!\bigcirc},
		$$
		and $(L',p')$ is an {\em  inverse of $(L,p)$}.
	\end{defn}
	Note that the notion of the unit is different from the singular connected summand. 
	The difference occurs when $L_{\bigcirc\!\!\!\bigcirc}$ is decomposed into 3 or more summands, and so a unit is a special kind of summand of $L_{\bigcirc\!\!\!\bigcirc}$ so that its degree is either 1, 2 or 3.
	
	While there is only one summand under the connected sum $\#$ of the identity element $L_\bigcirc$, which is $L_\bigcirc$ itself, 
	there are infinitely many units for degree 1 and 3 including all closures of tangles for resolutions, where each of them has infinitely many inverses.
	
	Nevertheless, we can say that there are no units of degree 2 except the identity. In order to clarify the argument, we need the following preparations.
	
%

	\begin{lem}\label{lem:mirrorpair}
		Let $L$ be a unit of degree 2. Then $L\in \mathcal{L}(\bigcirc\!\!\!\bigcirc)$ and $tb(L)=-2$.
	\end{lem}
	\begin{proof}
		Note first that $L_{\bigcirc\!\!\!\bigcirc}$ can be considered as the trivial $\theta_4$-curves and so is its factor $L$ as a $\theta_4$-curve. See \cite[Lemma~2.1]{M}. Therefore $\|L\|$ is the same as $\bigcirc\!\!\!\bigcirc$ up to twisting at vertices, 
		$
		\vcenter{\hbox{\includegraphics{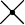}}} ~~\longleftrightarrow~~ 		\vcenter{\hbox{\includegraphics{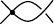}}}  
		$,
		which implies that $\|L\|$ can be represented by a four strand braid $\beta_L$ on the sphere. It is known that $\beta_L$ is well-defined up to conjugate and the following two moves.\footnote{These moves correspond to the generators for {\em Mexican plaits} as described in \cite{Mu}.}
		
		\[
		\xymatrix@R=1pc{
			\vcenter{\hbox{
\begingroup%
  \makeatletter%
  \providecommand\color[2][]{%
    \errmessage{(Inkscape) Color is used for the text in Inkscape, but the package 'color.sty' is not loaded}%
    \renewcommand\color[2][]{}%
  }%
  \providecommand\transparent[1]{%
    \errmessage{(Inkscape) Transparency is used (non-zero) for the text in Inkscape, but the package 'transparent.sty' is not loaded}%
    \renewcommand\transparent[1]{}%
  }%
  \providecommand\rotatebox[2]{#2}%
  \ifx\svgwidth\undefined%
    \setlength{\unitlength}{52.8bp}%
    \ifx\svgscale\undefined%
      \relax%
    \else%
      \setlength{\unitlength}{\unitlength * \real{\svgscale}}%
    \fi%
  \else%
    \setlength{\unitlength}{\svgwidth}%
  \fi%
  \global\let\svgwidth\undefined%
  \global\let\svgscale\undefined%
  \makeatother%
  \begin{picture}(1,0.99999934)%
    \put(0,0){\includegraphics[width=\unitlength,page=1]{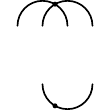}}%
    \put(0.45670995,0.58928485){\color[rgb]{0,0,0}\makebox(0,0)[lb]{\smash{$\beta$}}}%
    \put(0,0){\includegraphics[width=\unitlength,page=2]{beta_1.pdf}}%
  \end{picture}%
\endgroup%
}}\ar@{<->}[r] & \vcenter{\hbox{
\begingroup%
  \makeatletter%
  \providecommand\color[2][]{%
    \errmessage{(Inkscape) Color is used for the text in Inkscape, but the package 'color.sty' is not loaded}%
    \renewcommand\color[2][]{}%
  }%
  \providecommand\transparent[1]{%
    \errmessage{(Inkscape) Transparency is used (non-zero) for the text in Inkscape, but the package 'transparent.sty' is not loaded}%
    \renewcommand\transparent[1]{}%
  }%
  \providecommand\rotatebox[2]{#2}%
  \ifx\svgwidth\undefined%
    \setlength{\unitlength}{52.8bp}%
    \ifx\svgscale\undefined%
      \relax%
    \else%
      \setlength{\unitlength}{\unitlength * \real{\svgscale}}%
    \fi%
  \else%
    \setlength{\unitlength}{\svgwidth}%
  \fi%
  \global\let\svgwidth\undefined%
  \global\let\svgscale\undefined%
  \makeatother%
  \begin{picture}(1,0.99999934)%
    \put(0,0){\includegraphics[width=\unitlength,page=1]{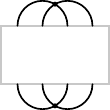}}%
    \put(0.45670995,0.4567091){\color[rgb]{0,0,0}\makebox(0,0)[lb]{\smash{$\beta$}}}%
  \end{picture}%
\endgroup%
}} \ar@{<->}[r]& \vcenter{\hbox{
\begingroup%
  \makeatletter%
  \providecommand\color[2][]{%
    \errmessage{(Inkscape) Color is used for the text in Inkscape, but the package 'color.sty' is not loaded}%
    \renewcommand\color[2][]{}%
  }%
  \providecommand\transparent[1]{%
    \errmessage{(Inkscape) Transparency is used (non-zero) for the text in Inkscape, but the package 'transparent.sty' is not loaded}%
    \renewcommand\transparent[1]{}%
  }%
  \providecommand\rotatebox[2]{#2}%
  \ifx\svgwidth\undefined%
    \setlength{\unitlength}{52.8bp}%
    \ifx\svgscale\undefined%
      \relax%
    \else%
      \setlength{\unitlength}{\unitlength * \real{\svgscale}}%
    \fi%
  \else%
    \setlength{\unitlength}{\svgwidth}%
  \fi%
  \global\let\svgwidth\undefined%
  \global\let\svgscale\undefined%
  \makeatother%
  \begin{picture}(1,0.99999934)%
    \put(0,0){\includegraphics[width=\unitlength,page=1]{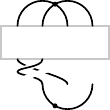}}%
    \put(0.45670995,0.58928485){\color[rgb]{0,0,0}\makebox(0,0)[lb]{\smash{$\beta$}}}%
    \put(0,0){\includegraphics[width=\unitlength,page=2]{beta_2.pdf}}%
  \end{picture}%
\endgroup%
}}\\
			\beta\sigma_1\sigma_3^{-1}\ar@{<->}[r]& \beta \ar@{<->}[r] & \beta\sigma_1\sigma_2\sigma_1
		}
		\]
	Then one can choose $\beta_L$ so that $\beta_L$ is $\sigma_1$-free and its standard projection is an alternating diagram with $C$ crossings\footnote{This process is exactly the same as the way how to obtain an {\em alternating normal form} for two bridge links. See \cite{Mu2} for detail.}.
	Moreover, there exists a pair $(R_{\eta_1},R_{\eta_2})$ of resolutions so that the number of crossings in $R_{\eta_1}(R_{\eta_2}(L,p),q)$ is exactly $C+2$.
		
		
		For an inverse $L'$ of $L$, since $L'$ is represented by $\beta_L^{-1}$, we have the mirror pair of reduced alternating diagrams for resolutions:
		$$
		K_1:=\|R_{\eta_1}(R_{\eta_2}(L_,p),q)\|,\qquad
		K_2:=\|R_{-\eta_1}(R_{-\eta_2}(-L',p'),q')\|,
		$$
		whose crossing numbers are exactly $C+2$.
		%
		Since both $K_1$ and $K_2$ are alternating links, by \cite[Corollary 1.4]{Ta}, the sum of maximal Thurston-Bennequin numbers $\overline{tb}$ of the mirror pair $(K_1,K_2)$ satisfies
		$$
		\overline{tb}(K_1)+\overline{tb}(K_2)= -c(K_i)-2,
		$$
		where $c(K)$ is the minimal crossing number of $K$.
		
		Now we obtain
		\begin{align*}
			-2=&tb(L_{\bigcirc\!\!\!\bigcirc})=tb(L)+tb(L')+2\\
			=&tb(R_{\eta_1}(R_{\eta_2}(L,p),q))+tb(R_{-\eta_1}(R_{-\eta_2}(-L',p'),q'))+2\\
			\leq&\overline{tb}(K_1)+\overline{tb}(K_2)+2\\
			=&-c(K_i)=-C-2.
		\end{align*}
		Therefore $C=0$ and so $\beta_L=1\in B_4(S^2)$, which implies that both $L$ and $L'$ are contained in $\mathcal{L}(\bigcirc\!\!\!\bigcirc)$. Since $L_{\bigcirc\!\!\!\bigcirc}$ has the maximal $tb$, neither $tb(L)$ nor $tb(L')$ can not exceed $tb(L_{\bigcirc\!\!\!\bigcirc})=-2$. Therefore they must be equal to $-2$ as claimed.
	\end{proof}

\begin{thm}\label{thm:L00}
A Legendrian singular link $L\in \mathcal{L}(\bigcirc\!\!\!\bigcirc)$ has the maximal $tb$, i.e., $tb(L)=-2$, if and only if $L=L_{\bigcirc\!\!\!\bigcirc}$.
\end{thm}		
\begin{proof}
The ``if'' part is obvious.

		Assume that $L\in\mathcal{L}(\bigcirc\!\!\!\bigcirc)$ with $tb(L)=-2$. We can take a sphere $S\subset\R^3$ containing $L$. We then may assume that $S_{\xi_{rot}}$ is Morse-Smale and has two elliptic singularities at $p$ and $q$.
		
		Let $a_1, a_2, a_3, a_4$ be Legendrian arcs which are the closures of $L\setminus\{p,q\}$ and are labelled by $\sigma(L,p)$.
		Now let $t_i$ be a (clock-wise) twisting number of contact planes from $p$ to $q$ along $a_i$ with respect to $S$,
		then $t_i$ should be contained in $\frac{1}{2}\Z$.
		
		Note that $R_{+}(R_{-}(L_,p),q)$, $R_{0}(R_{0}(L_,p),q)$, and ${R_{0}(R_{0}(\overline{L}_,p),q)}$ are the two components unlinks with $tb=-2$, where $\overline{L}$ is the same as $L$ with one component reversed.
		This implies that each component of them is the unknot with $tb=-1$.
		By the definition of Thurston-Bennequin number, we have $t_i+t_j=-1$ where $i\neq j$ and $i,j\in\{1,2,3,4\}$.
		
		The only possible case is that $t_i=-\frac{1}{2}$ for $i=1,\dots,4$.
		Then we can perturb $S$ in such a way that $S_{\xi_{rot}}|_{a_i}$ has singularities only at $p, q$.
		By Giroux elimination process, we may assume that $S$ has only two elliptic singularities at $p,q$ which is contactomorphic to $S_{std}$. 
		This completes the theorem.
	\end{proof}

\begin{proof}[Proof of Theorem \ref{thm:unit}]
		As mentioned right after Definition \ref{def:identity}, the identity is $L_{\bigcirc\!\!\!\bigcirc}$.
		
		Let $L$ be a unit of degree $2$, then by Lemma \ref{lem:mirrorpair}, we have $tb(L)=-2$ and $L\in\mathcal{L}(\bigcirc\!\!\!\bigcirc)$.
		Hence Theorem \ref{thm:L00} implies that $L=L_{\bigcirc\!\!\!\bigcirc}$. 
\end{proof}

		\begin{rmk}
			As shown in Table~\ref{table:ambiguities}, there are infinitely many singular links of degree 2 in $\SK$ which decompose $\bigcirc\!\!\!\bigcirc$, and so which are units.
%
			Indeed, 
			for any given rational tangle we can make a corresponding degree 2 unit.
		\end{rmk}

	\section{Diagrammatic interpretations of singular connected sum}
	\subsection{Tangle representatives}
	A {\em Legendrian singular tangle $T$} is an oriented Legendrian immersion 
	$$
	T:\left(2I\coprod nS^1,2\partial I\right)\to(B_{std},\{\mathbf{0}_x,\mathbf{0}_y,-\mathbf{0}_x,-\mathbf{0}_y\})
	$$
	such that $T$ has only double point singularities in the interior and intersects $\partial B_{std}$ perpendicularly at $\{\mathbf{0}_x,\mathbf{0}_y,-\mathbf{0}_x,-\mathbf{0}_y\}$ matching the orientation with $x$ and $y$-axes.
	
	Then the {\em (singular) closure $\widehat T\in\LSK$ of $T$} is obtained by 
	$$
	(S^3,\widehat T)=(B_{std},T)\coprod_\phi(B_{std},I_x\cup I_y),
	$$
	where $\phi$ is a diffeomorphism on $S_{std}$ preserving the characteristic foliation such that
	$$\phi((\mathbf{0}_x,\mathbf{0}_y,-\mathbf{0}_x,-\mathbf{0}_y))=-\sigma(I_x\cup I_y,\mathbf{0})$$
	as before. See Figure~\ref{fig:tangle} for a pictorial definition of a tangle closure.
	
	\begin{figure}[ht]
		$$
		\widehat{\vcenter{\hbox{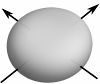}}}=
		\vcenter{\hbox{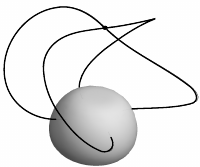}}
		$$
		\caption{A Legendrian Tangle $T$ and its closure $\widehat T$}
		\label{fig:tangle}
	\end{figure}
	
	We say that two tangles $T_1$ and $T_2$ are {\em equivalent} if they are contact isotopic relative to their boundaries, or equivalently, two pairs $(\widehat T_1,\mathbf{0})$ and $(\widehat T_2,\mathbf{0})$ are contact isotopic. If there is a contact isotopy between closures, then we can modify the isotopy so that $\mathbf{0}$ and its neighborhood are fixed during the isotopy and the support is contained in $\mathring{B}_{std}$. The other direction is clear.
	
	For a Legendrian singular tangle $T$, let $(\widehat T,\mathbf{0})$ be the closure of $T$ and $\phi:S^3\to S^3$ be a contactomophism such that $\phi(\mathbf{0})=\infty\in S^3$. Then 
	\[
	\phi(T)\setminus\{\infty\}\subset (S^3\setminus\{\infty\},\xi_{std})=(\R^3,\xi_0),
	\]
	and the {\em front projection $\pi_F(T,\phi)$} of $T$ with respect to $\phi$ is defined as
	\[
	\pi_F(T,\phi):=\pi_F(\phi(T)\setminus\{\infty\}),
	\]
	and we call $\phi$ a {\em way of the projection} of $F$.
	
	Similar to the front projection of a singular point, we obtain four types of front projections of a tangle $T$ with respect to some contactomorphisms $\phi_N,\phi_E,\phi_S$ and $\phi_W$ according to the local pictures near $\mathbf{0}$ as depicted in Figure~\ref{fig:singularpoints}.
	Intuitively, these also correspond to the ways in which the tangles are projected. Figure~\ref{fig:fronttangle} shows the corresponding projections.
	
	\begin{figure}[ht]
		\begin{align*}
			\pi_F(T,\phi_N)=\!\!\!\!\!\vcenter{\hbox{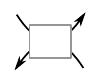}}\qquad&
			\pi_F(T,\phi_E)=\!\!\!\!\!\vcenter{\hbox{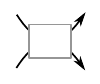}}\\ 
			\pi_F(T,\phi_S)=\!\!\!\!\!\vcenter{\hbox{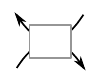}}\qquad&
			\pi_F(T,\phi_W)=\!\!\!\!\!\vcenter{\hbox{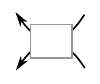}}
		\end{align*}
		\caption{Front projections for a Legendrian tangle $T$}
		\label{fig:fronttangle}
	\end{figure}
	
	It is obvious that if two front projections of tangles are connected by a sequence of Reidemeister moves $\rm(I)\sim(VI)$, then they are equivalent. However all Reidemeister moves preserve the orientation at the boundary of tangle, so we need a global move $\rm(VI_T)$ as depicted in Figure~\ref{fig:tanglemove}, which changes the way of the projection and therefore the configuration at the boundary. In the closure $\widehat T$, this move is nothing but a Reidemeister move $\rm(VI)$ at $\mathbf{0}$.
	
	\begin{figure}[ht]
		$$
		\xymatrix{
			\vcenter{\hbox{\input{F_tangle_RR.pdf_tex}}}\ar@{<->}[rr]^-{\rm(VI_T)}& &
			\vcenter{\hbox{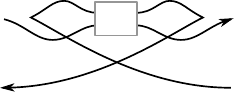}}
		}
		$$
		\caption{A global move $\rm(VI_T)$ for tangles}
		\label{fig:tanglemove}
	\end{figure}
	
	\begin{lem}
		Let $T_1$ and $T_2$ be two Legendrian singluar tangles. Suppose that $D_1=\pi_F(T_1,\phi_1)$ and $D_2=\pi_F(T_2,\phi_2)$ are regular front projections with respect to some $\phi_1$ and $\phi_2$, respectively.
		Then $T_1$ and $T_2$ are equivalent if and only if $D_1$ and $D_2$ are connected by a sequence of (local) Reidemeister moves $\rm(I)\sim(VI)$, and the global move $\rm(VI_T)$.
	\end{lem}
	\begin{proof}
		The ``if'' part is obvious since the moves $\rm(I)\sim(VI)$ and $\rm(VI_T)$ can be realized via contact isotopy.
		
		Conversely, suppose that $T_1$ and $T_2$ are equivalent.
		Then by taking $\rm(VI_T)$ several times, we may assume that both $D_i$'s have the same configurations at the boundary.
		Since two tangles are contact isotopic relative to boundary, two diagrams $D_i$'s are connected by Reidemeister moves $\rm(I)\sim(VI)$ inside the standard 3-ball by Proposition~\ref{prop:frontmove}.
	\end{proof}
	
	\begin{lem}\label{lem:tangleclosure}
		Let $L\in \LSK$ and $p\in\P(L)$. Then there exists a tangle $T_{(L,p)}$ whose closure $\widehat T_{(L,p)}$ is equivalent to $L$, where the designated singular point $p$ corresponds to $\mathbf{0}$ of $\widehat T_{(L,p)}$.
	\end{lem}
	\begin{proof}
		Let $B_p$ be a standard neighborhood of $p$. Then the complement $T_{(L,p)}=L\cap(S^3\setminus B_p)$ satisfies the tangle conditions.
		Hence by the definition of a closure, $(L,p)$ is nothing but $(\widehat T_{(L,p)},\mathbf{0})$.
	\end{proof}
	
	Note that $\mathbf{0}\in B_{std}$ is a singular point of $\widehat T$ produced by the closure, and
	$\widehat T$ can be isotoped so that $T$ becomes arbitrarily small.
	This means that $\widehat T$ can be isotoped into a union of $L_{\bigcirc\!\!\!\bigcirc}$ and a small neighborhood $B_\infty$ of a singular point $\infty\in L_{\bigcirc\!\!\!\bigcirc}$.
	
	In Figure \ref{fig:tanglesumidentity}, the $xy$-plane in the first figure  corresponds to the dotted line in the second figure.
	There is a contact isotopy between second and third front projections which maps the horizontal dotted line to the vertical dotted line.
	Thus by this isotopy, the small ball $B_{\infty}$ containing $T$ is transformed to a neighborhood of $\infty$ in the each diagrams.
	
	Similarly, since two singular points $\{\mathbf{0},\infty\}$ in $L_{\bigcirc\!\!\!\bigcirc}$ are equivalent in the sense that there is a contact isotopy interchanging those points,
	we may also change the role of given tangle $T$ and singular point $\mathbf{0}$ in $\widehat T$.
	Therefore by Lemma~\ref{lem:tangleclosure}, $\widehat T$ has two special types of front projections, called {\em left and right normal forms at $\mathbf{0}$}, as depicted in Figure~\ref{fig:normalform}. 
	\begin{figure}[ht]
		{\scriptsize
			$$
			\vcenter{\hbox{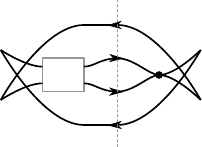}}=
			\widehat{\vcenter{\hbox{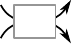}}}=
			\vcenter{\hbox{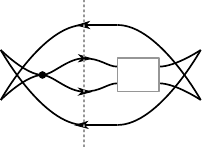}}
			$$
		}
		\caption{Left and right normal form of $\widehat T$}
		\label{fig:normalform}
	\end{figure}
	We can also consider left and right normal forms for each front projection of $T$ depicted in Figure~\ref{fig:fronttangle}.

	Here are the relationships between singular and {\em regular} closures of given tangle.
	The {\em regular} closures $D(T)$ and $N_\pm(T)$ of $T$ in the front projection are as depicted in Figure~\ref{fig:twoclosures}. 
	Note that these closures are mimics of the {\em denominator} and {\em numerator} closures of rational tangles. 
	Then by definition, $D(T)= R_0\left(\widehat T, \mathbf{0}\right)$ and $N_{\pm}(T)=R_{\pm}\left(\widehat T,\mathbf{0}\right)$. Hence, the singular closure can be thought of as a generalization of the regular closure of tangles.
	
	\begin{figure}[ht]
		\begin{align*}
			D\left(\vcenter{\hbox{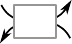}}\right)&=\vcenter{\hbox{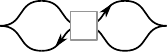}}=R_0\left(\widehat T,\mathbf{0}\right),\\
			N_+\left(\vcenter{\hbox{\input{F_tangle.pdf_tex}}}\right)&=\vcenter{\hbox{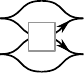}}=R_+\left(\widehat T,\mathbf{0}\right),\\
			N_-\left(\vcenter{\hbox{\input{F_tangle2.pdf_tex}}}\right)&=\vcenter{\hbox{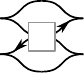}}=R_-\left(\widehat T,\mathbf{0}\right).
		\end{align*}
		\caption{Numerator and denominator closures of $T$}
		\label{fig:twoclosures}
	\end{figure}
	
	In particular, the simplest singular tangle $\vcenter{\hbox{\includegraphics{char_F_singularcrossingRL_dot.pdf}}}$ under $D$ and $N_{\pm}$ corresponds to 
	the simplest Legendrian singular knot $L_0$ and links $L_{\pm}$ having $tb(L_{\pm})=\pm1-2$, whose link types $K_{\pm}=\|L_{\pm}\|$ are called the {\em $\pm$-pinched Hopf links}.
	Notice that $L_\eta=R_\eta(L_{\bigcirc\!\!\!\bigcirc},\mathbf{0})$ and $R_{-\eta}(L_\eta)$ is the Legendrian unlink for any $\eta\in\{+,-,0\}$.
	\begin{figure}[ht]
		$$
		L_{+}=\vcenter{\hbox{\includegraphics{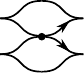}}},\qquad
		L_{-}=\vcenter{\hbox{\includegraphics{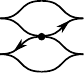}}}
		$$
		\caption{The pinched Hopf links $L_{\pm}$}
		\label{fig:pinchedhopflinks}
	\end{figure}
	
	\subsection{Singular connected sum in the projection}
	Let $(L_1,p)$ and $(L_2,q)$ be pairs of Legendrian singular links and singular points.
	Suppose that the front projections $\pi_F(L_1)$ and $\pi_F(L_2)$ are of left and right normal forms at $p$ and $q$, respectively.
	Then the front projection $\pi_F((L_1,p)\otimes(L_2,q))$ of singular connected sum of $(L_1,p)$ and $(L_2,q)$ is the diagrammatic concatenation of the left part of $(L_1,p)$ and the right part of the $(L_2,q)$ as depicted in Figure ~\ref{fig:connectedsumfront}.
	Note that the above diagrammatic gluing of the pair of nearby points coincides with the condition $\sigma(L_1,p)=-\sigma(L_2,q)$.
	For the reader's convenience we provide an abstract diagram for the singular connected sum in $\R^3$ in Figure~\ref{fig:connectedsumin3d}.
	\begin{figure}[ht]
		{\scriptsize
			$$
			\vcenter{\hbox{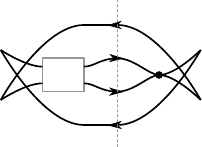}}
			\otimes
			\vcenter{\hbox{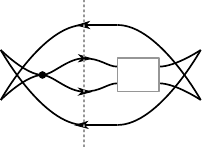}}=
			\vcenter{\hbox{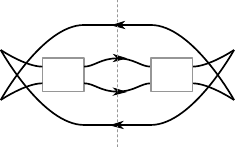}}
			$$
		}
		\caption{Singular connected sum in the front projection}
		\label{fig:connectedsumfront}
	\end{figure}

	\begin{figure}[ht]
		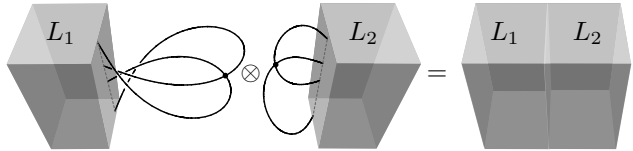
		\caption{Singular connected sum in $\R^3$}
		\label{fig:connectedsumin3d}
	\end{figure}
	
	The behaviour of the classical invariants $tb$ and $r$ under the singular connected sum are as follows.
	$$
	tb((L_1,p)\otimes(L_2,q))=tb(L_1)+tb(L_2)+2,\quad
	r((L_1,p)\otimes(L_2,q))=r(L_1)+r(L_2).
	$$
	Obviously, the set of singular points after a singular connected sum is as follows.
	$$
	\P((L_1,p)\otimes(L_2,q))= (\P(L_1)\setminus\{p\})\cup(\P(L_2)\setminus\{q\}).
	$$

	\subsection{Tangle replacements}
	We define an operation called the {\em tangle replacement} as follows. It is essentially same as the singular connected sum but is easier to describe.
	Indeed, normal forms are not necessary.
	
	\begin{defn}
		Let $L\in\LSK$, $p\in\P(L)$, and $T$ be a Legendrian tangle.
		The {\em tangle replacement $\langle (L,p),T\rangle$} of $L$ and $T$ is defined as the singular connected sum $(L,p)\otimes(\widehat T, \mathbf{0})$.
	\end{defn}
	
	Then in a diagrammatic view, this is nothing but the replacement of a small neighborhood of $p$ in $L$ with the given tangle $T$ with the obvious matching condition at the boundary.
	Note that there is no problem in realizing the resulting diagram as a Legendrian singular link since we can make $T$ sufficiently small. 
	Moreover the diagrammatic replacement is valid for both the front and the Lagrangian projections.
	See Figure~\ref{fig:tanglesum}.
	
	\begin{figure}[ht]
		$$
		\left\langle
		\vcenter{\hbox{\includegraphics{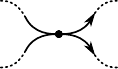}}},
		\vcenter{\hbox{\input{F_tangle.pdf_tex}}}
		\right\rangle=
		\vcenter{\hbox{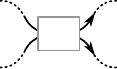}}
		$$
		\caption{A tangle replacement of Legendrian singular link $L$ with a tangle $T$}
		\label{fig:tanglesum}
	\end{figure}
	
	Moreover, the resolutions $R_\pm$, $R_0$, 
	can be interpreted as special cases of tangle replacements, described below\footnote{The $\infty$-resolution $R_\infty$ is an {\em unoriented tangle replacement} with $\vcenter{\hbox{\includegraphics[scale=0.3]{F_singularcrossingRL_zero.pdf}}}$. See \S\ref{sec:unori}.}.
	$$
	R_0=\left\langle - , \vcenter{\hbox{\includegraphics[scale=0.5]{F_singularcrossingRL_zero.pdf}}}\right\rangle,
	\quad R_+=\left\langle -,\vcenter{\hbox{\includegraphics[scale=0.5]{F_singularcrossingRL_pos.pdf}}}\right\rangle,
	\quad R_-=\left\langle-,\vcenter{\hbox{\includegraphics[scale=0.5]{F_singularcrossingRL_neg.pdf}}}\right\rangle.
	$$
	
	Equivalently, by Figure~\ref{fig:twoclosures}, we have the dual descriptions of resolutions as follows.
	$$
	R_\eta\left(\widehat -,\mathbf{0}\right)=\langle L_\eta, -\rangle,
	\quad
	\eta\in\{+,-,0\}
	$$
	
	In general, for each tangle $T$, a tangle replacement $\langle - ,T\rangle$ gives us an operation on Legendrian singular links which may be used to produce new invariants and to distinguish Legendrian singular links which are indistinguishable even by the Legendrian singular link types of all resolutions.
	
	\section{Applications}

	\subsection{Proof of Theorem \ref{thm:pairoflsk}}
	Now we are ready to use the singular connected sum to distinguish two Legendrian singular links described in Example~\ref{ex:final}.
	
	Suppose $L_e$ and $-L_e$ are the same in $\LSK$.
	Let $p$, $q$ be the singular points of $L_e$ and $-L_e$ respectively.
	Then any isotopy between them should map $p$ to $q$.
	Hence by the well-definedness of the singular connected sum, two singular connected sums $(L_e,p)\otimes(L,\mathbf{0})$ and $(-L_e,q)\otimes(L,\mathbf{0})$ are the same in $\LSK$ for any pair $(L,\mathbf{0})$ of Legendrian singular link $L$ and $\mathbf{0}\in\P(L)$.
	
	First, we choose $L$ of degree one as follows.
	{\scriptsize
		$$
\begingroup%
  \makeatletter%
  \providecommand\color[2][]{%
    \errmessage{(Inkscape) Color is used for the text in Inkscape, but the package 'color.sty' is not loaded}%
    \renewcommand\color[2][]{}%
  }%
  \providecommand\transparent[1]{%
    \errmessage{(Inkscape) Transparency is used (non-zero) for the text in Inkscape, but the package 'transparent.sty' is not loaded}%
    \renewcommand\transparent[1]{}%
  }%
  \providecommand\rotatebox[2]{#2}%
  \ifx\svgwidth\undefined%
    \setlength{\unitlength}{112.7bp}%
    \ifx\svgscale\undefined%
      \relax%
    \else%
      \setlength{\unitlength}{\unitlength * \real{\svgscale}}%
    \fi%
  \else%
    \setlength{\unitlength}{\svgwidth}%
  \fi%
  \global\let\svgwidth\undefined%
  \global\let\svgscale\undefined%
  \makeatother%
  \begin{picture}(1,0.64529725)%
    \put(0,0){\includegraphics[width=\unitlength]{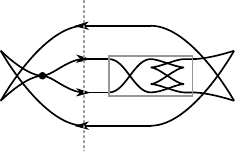}}%
    \put(0.39352263,0.5710759){\color[rgb]{0,0,0}\makebox(0,0)[lb]{\smash{$-\mathbf{0}_x$}}}%
    \put(0.39352263,0.42910607){\color[rgb]{0,0,0}\makebox(0,0)[lb]{\smash{$\mathbf{0}_y$}}}%
    \put(0.39352263,0.18065886){\color[rgb]{0,0,0}\makebox(0,0)[lb]{\smash{$\mathbf{0}_x$}}}%
    \put(0.39352263,0.03868903){\color[rgb]{0,0,0}\makebox(0,0)[lb]{\smash{$-\mathbf{0}_y$}}}%
    \put(0.14507542,0.25164378){\color[rgb]{0,0,0}\makebox(0,0)[lb]{\smash{$\mathbf{0}$}}}%
  \end{picture}%
\endgroup%

		$$
	}
	Now we find the left normal forms of both $L_e$ and $-L_e$ at each singular point.
	{\scriptsize
		$$
\begingroup%
  \makeatletter%
  \providecommand\color[2][]{%
    \errmessage{(Inkscape) Color is used for the text in Inkscape, but the package 'color.sty' is not loaded}%
    \renewcommand\color[2][]{}%
  }%
  \providecommand\transparent[1]{%
    \errmessage{(Inkscape) Transparency is used (non-zero) for the text in Inkscape, but the package 'transparent.sty' is not loaded}%
    \renewcommand\transparent[1]{}%
  }%
  \providecommand\rotatebox[2]{#2}%
  \ifx\svgwidth\undefined%
    \setlength{\unitlength}{112.7bp}%
    \ifx\svgscale\undefined%
      \relax%
    \else%
      \setlength{\unitlength}{\unitlength * \real{\svgscale}}%
    \fi%
  \else%
    \setlength{\unitlength}{\svgwidth}%
  \fi%
  \global\let\svgwidth\undefined%
  \global\let\svgscale\undefined%
  \makeatother%
  \begin{picture}(1,0.64529725)%
    \put(0,0){\includegraphics[width=\unitlength]{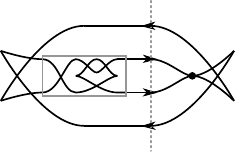}}%
    \put(0.78393966,0.25164376){\color[rgb]{0,0,0}\makebox(0,0)[lb]{\smash{$p$}}}%
    \put(0.53549246,0.5710759){\color[rgb]{0,0,0}\makebox(0,0)[lb]{\smash{$p_x$}}}%
    \put(0.5,0.42910607){\color[rgb]{0,0,0}\makebox(0,0)[lb]{\smash{$-p_y$}}}%
    \put(0.5,0.18065886){\color[rgb]{0,0,0}\makebox(0,0)[lb]{\smash{$-p_x$}}}%
    \put(0.53549246,0.03868903){\color[rgb]{0,0,0}\makebox(0,0)[lb]{\smash{$p_y$}}}%
  \end{picture}%
\endgroup%
\qquad
\begingroup%
  \makeatletter%
  \providecommand\color[2][]{%
    \errmessage{(Inkscape) Color is used for the text in Inkscape, but the package 'color.sty' is not loaded}%
    \renewcommand\color[2][]{}%
  }%
  \providecommand\transparent[1]{%
    \errmessage{(Inkscape) Transparency is used (non-zero) for the text in Inkscape, but the package 'transparent.sty' is not loaded}%
    \renewcommand\transparent[1]{}%
  }%
  \providecommand\rotatebox[2]{#2}%
  \ifx\svgwidth\undefined%
    \setlength{\unitlength}{112.7bp}%
    \ifx\svgscale\undefined%
      \relax%
    \else%
      \setlength{\unitlength}{\unitlength * \real{\svgscale}}%
    \fi%
  \else%
    \setlength{\unitlength}{\svgwidth}%
  \fi%
  \global\let\svgwidth\undefined%
  \global\let\svgscale\undefined%
  \makeatother%
  \begin{picture}(1,0.64529725)%
    \put(0,0){\includegraphics[width=\unitlength]{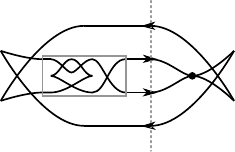}}%
    \put(0.78393966,0.25164376){\color[rgb]{0,0,0}\makebox(0,0)[lb]{\smash{$q$}}}%
    \put(0.53549246,0.5710759){\color[rgb]{0,0,0}\makebox(0,0)[lb]{\smash{$q_x$}}}%
    \put(0.5,0.42910607){\color[rgb]{0,0,0}\makebox(0,0)[lb]{\smash{$-q_y$}}}%
    \put(0.5,0.18065886){\color[rgb]{0,0,0}\makebox(0,0)[lb]{\smash{$-q_x$}}}%
    \put(0.53549246,0.03868903){\color[rgb]{0,0,0}\makebox(0,0)[lb]{\smash{$q_y$}}}%
  \end{picture}%
\endgroup%

		$$
	}
	The Legendrian knots $(L_e,p)\otimes(L,\mathbf{0})$, $(-L_e,q)\otimes(L,\mathbf{0})$ are as follows.
	{\scriptsize
		$$
\begingroup%
  \makeatletter%
  \providecommand\color[2][]{%
    \errmessage{(Inkscape) Color is used for the text in Inkscape, but the package 'color.sty' is not loaded}%
    \renewcommand\color[2][]{}%
  }%
  \providecommand\transparent[1]{%
    \errmessage{(Inkscape) Transparency is used (non-zero) for the text in Inkscape, but the package 'transparent.sty' is not loaded}%
    \renewcommand\transparent[1]{}%
  }%
  \providecommand\rotatebox[2]{#2}%
  \ifx\svgwidth\undefined%
    \setlength{\unitlength}{144.7bp}%
    \ifx\svgscale\undefined%
      \relax%
    \else%
      \setlength{\unitlength}{\unitlength * \real{\svgscale}}%
    \fi%
  \else%
    \setlength{\unitlength}{\svgwidth}%
  \fi%
  \global\let\svgwidth\undefined%
  \global\let\svgscale\undefined%
  \makeatother%
  \begin{picture}(1,0.50259157)%
    \put(0,0){\includegraphics[width=\unitlength]{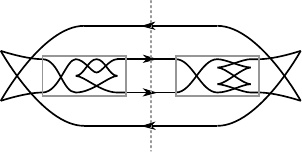}}%
    \put(0.5276434,0.44478406){\color[rgb]{0,0,0}\makebox(0,0)[lb]{\smash{$-\mathbf{0}_x$}}}%
    \put(0.5276434,0.33421046){\color[rgb]{0,0,0}\makebox(0,0)[lb]{\smash{$\mathbf{0}_y$}}}%
    \put(0.5276434,0.14070666){\color[rgb]{0,0,0}\makebox(0,0)[lb]{\smash{$\mathbf{0}_x$}}}%
    \put(0.5276434,0.03013306){\color[rgb]{0,0,0}\makebox(0,0)[lb]{\smash{$-\mathbf{0}_y$}}}%
    \put(0.4170698,0.44478406){\color[rgb]{0,0,0}\makebox(0,0)[lb]{\smash{$p_x$}}}%
    \put(0.3894264,0.33421046){\color[rgb]{0,0,0}\makebox(0,0)[lb]{\smash{$-p_y$}}}%
    \put(0.3894264,0.14070666){\color[rgb]{0,0,0}\makebox(0,0)[lb]{\smash{$-p_x$}}}%
    \put(0.4170698,0.03013306){\color[rgb]{0,0,0}\makebox(0,0)[lb]{\smash{$p_y$}}}%
  \end{picture}%
\endgroup%
\qquad 
\begingroup%
  \makeatletter%
  \providecommand\color[2][]{%
    \errmessage{(Inkscape) Color is used for the text in Inkscape, but the package 'color.sty' is not loaded}%
    \renewcommand\color[2][]{}%
  }%
  \providecommand\transparent[1]{%
    \errmessage{(Inkscape) Transparency is used (non-zero) for the text in Inkscape, but the package 'transparent.sty' is not loaded}%
    \renewcommand\transparent[1]{}%
  }%
  \providecommand\rotatebox[2]{#2}%
  \ifx\svgwidth\undefined%
    \setlength{\unitlength}{144.7bp}%
    \ifx\svgscale\undefined%
      \relax%
    \else%
      \setlength{\unitlength}{\unitlength * \real{\svgscale}}%
    \fi%
  \else%
    \setlength{\unitlength}{\svgwidth}%
  \fi%
  \global\let\svgwidth\undefined%
  \global\let\svgscale\undefined%
  \makeatother%
  \begin{picture}(1,0.50259157)%
    \put(0,0){\includegraphics[width=\unitlength]{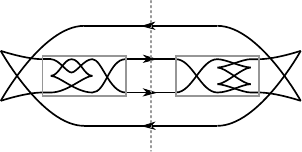}}%
    \put(0.5276434,0.44478406){\color[rgb]{0,0,0}\makebox(0,0)[lb]{\smash{$-\mathbf{0}_x$}}}%
    \put(0.5276434,0.33421046){\color[rgb]{0,0,0}\makebox(0,0)[lb]{\smash{$\mathbf{0}_y$}}}%
    \put(0.5276434,0.14070666){\color[rgb]{0,0,0}\makebox(0,0)[lb]{\smash{$\mathbf{0}_x$}}}%
    \put(0.5276434,0.03013306){\color[rgb]{0,0,0}\makebox(0,0)[lb]{\smash{$-\mathbf{0}_y$}}}%
    \put(0.4170698,0.44478406){\color[rgb]{0,0,0}\makebox(0,0)[lb]{\smash{$q_x$}}}%
    \put(0.3894264,0.33421046){\color[rgb]{0,0,0}\makebox(0,0)[lb]{\smash{$-q_y$}}}%
    \put(0.3894264,0.14070666){\color[rgb]{0,0,0}\makebox(0,0)[lb]{\smash{$-q_x$}}}%
    \put(0.4170698,0.03013306){\color[rgb]{0,0,0}\makebox(0,0)[lb]{\smash{$q_y$}}}%
  \end{picture}%
\endgroup%

		$$
	}
	
	Alternatively, we can interpret the above resulting diagrams in terms of the tangle replacement.
	Let us consider a tangle $T$ satisfying $(\widehat T,\mathbf{0})=(L,*)$ as follows.
	$$
	\includegraphics{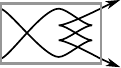}
	$$
	Finally, the resulting Legendrian knots are given as follows.
\[	\begin{array}{rcl}
	\langle(L_e,p),T\rangle&=\quad \vcenter{\hbox{\includegraphics{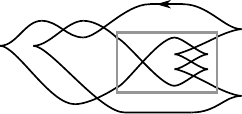}}}&=\quad\vcenter{\hbox{\includegraphics{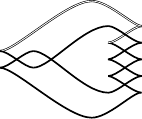}}}\\
	\langle(-L_e,p),T\rangle&=\quad \vcenter{\hbox{\includegraphics{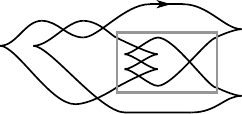}}}&=\quad \vcenter{\hbox{\includegraphics{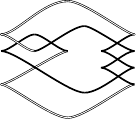}}}
	\end{array}
	\]
	The second equality for $\langle(L_e,p),T\rangle$ is given by one translation move\footnote{Legendrian isotopies can be interpreted in terms of combinatorial moves in the grid diagram: translations, commutations, (de)stabilizations, see \cite{NTh}.}, while the second equality for $\langle(-L_e,p),T\rangle$ needs two translation moves on doubled arcs.
	The topological knot types of the resulting Legendrian knots are same as $m(6_1)$.
	But their Poincar{\'e}-Chekanov polynomials, introduced in \cite{Ch}, are known to be different.
	$$
	P_{(L_e,p)\otimes(L,\mathbf{0})}(t)=t^{-3}+t+t^{3},\quad
	P_{(-L_e,q)\otimes(L,\mathbf{0})}(t)=t^{-1}+2t.
	$$
	Hence the resulting Legendrian knots are not the same in $\LK$, and therefore $L_e\neq-L_e$ in $\LSK$ which proves Theorem \ref{thm:pairoflsk}. 
	In other words, the singular knot type $\|L_e\|$ is $\{\mathcal{R}^m\}$-nonsimple.
	
	\subsection{Doubles in \texorpdfstring{$\LSK$}{Legendrian singular links} and Legendrian contact homology}
	
	For a given $L\in\LSK$, by virtue of the singular connected sum, we also have a Legendrian link from $L$ not obtained by the resolutions nor by a specific tangle choice. 
	
	Let $\P(L)=\{p_1,p_2,\dots,p_k\}$ and 
	consider $(L,p_i)\otimes(L,p_i)$ for each $i$.
	Let $B_{p_i}$ be the standard neighborhood of $p_i$ and $\phi_i:(\partial B_{p_i},\sigma(L,p_i))\to(\partial B_{p_i},-\sigma(L,p_i))$ be the gluing map defined as before.
	
	Now we introduce a {\em multiple singular connected sum} as follows.
	Consider two copies of $S^3\setminus(\coprod_i \mathring{B}_{p_i})$. Then we obtain $\#^{k-1}(S^2\times S^1)$ by gluing them via the map $\phi=\coprod_i \phi_i$ which admits the unique tight contact structure. The Legendrian link
	$$
	\mathcal{D}(L)=\left( L\setminus\coprod_i(L\cap\mathring{B}_{p_i})\right)\coprod_{\phi}\left( L\setminus\coprod_i(L\cap\mathring{B}_{p_i})\right)\subset \#^{k-1}(S^2\times S^1)
	$$
	is called a {\em double of $L$}. By the construction, $\mathcal{D}(L)$ has no singular points while the ambient contact manifold becomes a bit complicated.
	
	Recently a combinatorial description of the Legendrian contact homology algebra (DGA) of Legendrian links in $\#^{m}(S^2\times S^1)$ was developed in \cite{EN}. So we can assign an algebraic invariant, the DGA of $\mathcal{D}(L)$, to $L$.
	It would be interesting to investigate the relation between the DGA of $\mathcal{D}(L)$ and the DGAs of its resolutions $\mathcal{R}(L)$.
	
	Especially when $L$ is of degree one, we still have $\mathcal{D}(L)$ in $(S^3,\xi_{std})$.
	So we can use the ordinary Legendrian link invariants to study $L$. 
	As an example, the Legendrian singular knot $L_e$ depicted in Figure~\ref{fig:0,4_1} has the following front diagram in Figure~\ref{fig:doubleL_e} which can be obtained by concatenating the front diagram of $L_e$ with itself.
	\begin{figure}[ht]
		{\scriptsize
			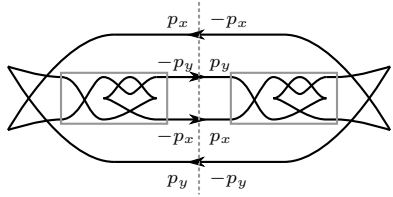}
		\caption{A front diagram of $\mathcal{D}(L_e)$}
		\label{fig:doubleL_e}
	\end{figure}
	One can check that $\mathcal{D}(L_e)$ is Legendrian isotopic to $\mathcal{D}(-L_e)$.
	
	\subsection{Splicing}\label{subsec:splicing}
	Another operation we can consider is {\em $\eta$-splicing $(L_1,p_1)*_\eta(L_2,p_2)$} of two Legendrian singular links $L_1$ and $L_2$ at regular points $p_1$ and $p_2$ for each $\eta\in\{+,-,0\}$.
	
	As mentioned before, the $(-\eta)$-resolution of $L_\eta$ at the singular point $\mathbf{0}$ gives a canonically ordered pair of Legendrian unknots $L_{\bigcirc,x}$ and $L_{\bigcirc,y}$ whose labels come from 
	$\sigma(L_\eta,\mathbf{0})$, that is, marked by $\mathbf{0}_x$ and $\mathbf{0}_y$.
	Then the $\eta$-splicing $(L_1,p_1)*_\eta(L_2,p_2)$ is defined by
	\begin{align*}
		(L_1,p_1)*_\eta(L_2,p_2)&=
		\left((L_1,p_1)\#(L_\eta,\mathbf{0}_x),\mathbf{0}_y\right)\#(L_2,p_2)\\
		&=
		(L_1,p_1)\#\left((L_\eta,\mathbf{0}_y)\#(L_2,p_2),\mathbf{0}_x\right).
	\end{align*}

	Indeed, all splicings are defined via connected sums and are therefore defined on $\SK$ as well. 
	
	It is important to note that the splicing operation is not commutative in general. 
	More precisely, this is because the triple $(L_\eta,\mathbf{0}_x,\mathbf{0}_y)$ is not the same as $(L_\eta,\mathbf{0}_y,\mathbf{0}_x)$.
	Indeed, $(L_2,p_2)*_\eta(L_1,p_1)$ is obtained from $(L_1,p_1)*_\eta(L_2,p_2)$ by performing the flip operation exactly once, and so
	they share many invariants such as (i) the classical invariants: $\SK$ type, $tb$ and $r$; and (ii) Legendrian link types of resolutions $\mathcal{R}$.
	
	\begin{figure}[ht]
		$$
		\vcenter{\hbox{\includegraphics{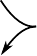}}}*_0
		\vcenter{\hbox{\includegraphics{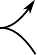}}}
		=\vcenter{\hbox{\includegraphics{F_LDCusp.pdf}}}\#
		\vcenter{\hbox{\includegraphics{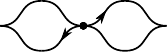}}}\#
		\vcenter{\hbox{\includegraphics{F_RUCusp.pdf}}}=
		\vcenter{\hbox{\includegraphics{F_singularcrossingRL_dot.pdf}}}
		$$
		$$
		\vcenter{\hbox{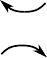}}
		=\vcenter{\hbox{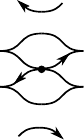}}
		=\vcenter{\hbox{\includegraphics{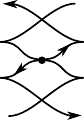}}},
		\qquad
		\vcenter{\hbox{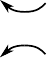}}
		=\vcenter{\hbox{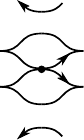}}
		=\vcenter{\hbox{\includegraphics{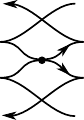}}}
		$$
		\caption{Front projections of splicings}
		\label{fig:splicing}
	\end{figure}
	
	As shown in Example~1, even the two splicings $(-*_0 L_\bigcirc)$ and $(L_\bigcirc*_0-)$ with the Legendrian unknot $L_\bigcirc$ are different in general. We call them {\em positive and negative singular stabilization} and denote them by $SS_{\pm}(L,p)$.
	The precise definitions are shown in Figure~\ref{fig:singularstab}.
	Topologically, these operations add a {\em singular kink} at $p$.
	We remark that singular stabilizations interpolate Legendrian links between given Legendrian singular links and their {\em transverse stabilizations} via $(+)$ and $(-)$-resolutions.
	
	\begin{figure}[ht]
		$$
		SS_+\left(\vcenter{\hbox{\includegraphics{F_LDCusp.pdf}}}\right)=
		\vcenter{\hbox{\includegraphics{F_LDCusp.pdf}}}*_0
		\vcenter{\hbox{\includegraphics{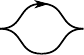}}}=
		\vcenter{\hbox{\includegraphics{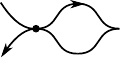}}}
		$$
		$$
		SS_-\left(\vcenter{\hbox{\includegraphics{F_RUCusp.pdf}}}\right)=
		\vcenter{\hbox{\includegraphics{L0.pdf}}}*_0
		\vcenter{\hbox{\includegraphics{F_RUCusp.pdf}}}=
		\vcenter{\hbox{\includegraphics{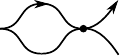}}}
		$$
		\caption{Positive and negative singular stabilizations $SS_{\pm}(L,p)$}
		\label{fig:singularstab}
	\end{figure}
	
	Moreover, by definition, the $0$-resolution $R_0((L_1,p_1)*_0(L_2,p_2),\mathbf{0})$ is a disjoint union $L_1\coprod L_2$, and the $(+)$-resolution $R_+((L_1,p_1)*_0(L_2,p_2),\mathbf{0})$ 
	is precisely a regular connected sum $(L_1,p_1)\#(L_2,p_2)$.
	Hence a $0$-splicing may be regarded as an intermediate state between the disjoint union and the connected sum. 
	
	On the other hand, $\eta$-splicings act like the inverses for the enhanced $(-\eta)$-resolutions $R_{-\eta}^m$ as follows.
	
	Let $L\in\LSK$ and $p\in\P(L)$. Suppose $R_{-\eta}(L,p)$ is a split link of 2-components for some $\eta\in\{+,-,0\}$.
	Then there is a sphere $S$ separating the components of $R_{-\eta}(L,p)$. Let $S_1$ and $S_2$ be parallel copies of $S$. By perturbing the $S_i$'s, we may assume that 
	$S_i$ intersects $L$ at 2 nearby points of $p$, and the arcs of $L$ contained between the $S_i$'s are precisely those in the standard neighborhood.
	The separating spheres $S_1$ and $S_2$ can be used to decompose $L$ into 3 connected summands, which coincide with those in the definition of $\eta$-splicing.
	Therefore $L$ is a $\eta$-splicing of two components of $R_{-\eta}(L,p)$ with the order coming from $\sigma(L,p)$.
	
	Conversely, let $L=(L_1,p_1)*_\eta(L_2,p_2)$ for some $\eta\in\{+,-,0\}$. Then we have $R_{-\eta}(L)=L_1\coprod L_2$, and lose the order of the splicing.
	However, the marking gives a label on each component of $L_1\coprod L_2$, which is equivalent to $\sigma(L,\mathbf{0})$, and we can recover $L$ from the $L_i$'s by using this order.
	Hence the enhanced $(-\eta)$-resolution $R_{-\eta}^m$ is the inverse of $*_\eta$ in both directions. In summary, we have the following theorem.
	
	\begin{thm}\label{thm:splicing}
		Let $K_1$ and $K_2$ be two singular links and $p_1$ and $p_2$ be regular points of $K_1$ and $K_2$, respectively.
		Then for each $\eta\in\{+,-,0\}$, the map 
		$$
		*_\eta:\mathcal{L}(K_1)\times\mathcal{L}(K_2)\cup\mathcal{L}(K_2)\times\mathcal{L}(K_1)
		\to\mathcal{L}( (K_1,p_1)*_\eta(K_2,p_2)),
		$$
		is bijective.
	\end{thm}
	Note that when $K_1=K_2$, then the union above is not disjoint.
	As a corollary, we have the following theorem.
	\begin{thm}\label{thm:simple}
		Let $K_1,K_2\in \SK$ be $\{f_1,\dots,f_k\}$-simple and $p_i\in K_i$ be a nonsingular point. 
		Then $(K_1,p_1)*_\eta(K_2,p_2)$ is $\{f_1(R^m_{-\eta}),\dots,f_k(R^m_{-\eta})\}$-simple.
	\end{thm}
	The proof is obvious from the above discussion, and we omit the proof.
	
	\begin{cor}\label{cor:Linftysimple}
		For each $\eta\in\{+,-,0\}$, $K_\eta$ is $\{\mathcal{R}\}$-nonsimple but $\{\mathbf{tb}(R_{-\eta}^m), \mathbf{r}(R_{-\eta}^m)\}$-simple.
	\end{cor}

	\appendix
	\section{Projection from \texorpdfstring{$S^3$}{3-sphere} to \texorpdfstring{$\R^3$}{Euclidean 3-space}}\label{sec:projection}
	In this section, we will give a concrete way to describe the singular connected sum.
	
	Recall that we regard $S^3$ as the unit sphere in $\mathbb{C}^2$ whoose coordinates are $(z,w)$, and the one-point compactification of $\R^3$.
	Let us regard $(0,-1)\in S^3$ as $\infty$ which compactifies $\R^3$.
	Then there is a well-known contactomorphism $\Phi:(S^3\setminus\{(0,-1)\},\xi_{std})\to(\R^3,\xi_{rot})$ as follows.
	$$
	\Phi(z,w)=\left(
	\frac z{1+w},\frac{\operatorname{Im}w}{|1+w|^2}
	\right)\in \mathbb{C}\times\R\simeq\R^3.
	$$
	
	Recall that $S^3$ can be decomposed into two solid tori separated by the torus $|z|^2=|w|^2=1/2$, and their core curve corresponds to the two circles $|z|=1$ and $|w|=1$ in $\mathbb{C}^2$.
	Then via the map $\Phi$, they are mapped to the unit circle $S^1_{xy}$ and $z$-axis in $\R^3$.
	
	We consider the rotations on $S^3$ as follows.
	The first one $Rot^0_t:S^3\to S^3$ comes from the rotation about the origin in $\mathbb{C}^2$ as follows.
	$$
	Rot^0_t(z)=z \cos t + w \sin t,\quad
	Rot^0_t(w)=-z \sin t + w \cos t.
	$$
	It is easy to check that $Rot^0_t$ gives a contact isotopy on $S^3$. That is, it preserves $\xi_{std}$ for all $t$.
	
	The other ones are the rotations $Rot^z_t$ and $Rot^w_t$ about $z$ and $w$ axes, respectively, which are defined by 
	$$
	R^z_t(z)=z, R^z_t(w)=w e^{it},\quad R^w_t(w)=w, R^w_t(z)=z e^{it}.
	$$
	We denote the push-forwards of $Rot^0_t$, $Rot^z_t$ and $Rot^w_t$ via $\Phi$ by the same notation.
	
	Recall the standard unit disc $D_{std}\subset \R_{xy}\subset\R^3$. 
	Then $D_{std}$ corresponds to when $w$ is real and positive in $S^3$, and the image of $D_{std}$ under $Rot^z_{\pm\pi/2}$ forms a sphere corresponding to when $w$ is purely imaginary, that is, $w=i y_2$ for $y_2\in[-1,1]$.
	Then $\Phi$ gives equations 
	$$
	r^2=\frac{1-y_2^2}{1+y_2^2}, \quad z=\frac{y_2}{1+y_2^2}
	$$
	satisfying $r^4+4z^2=1$, the defining equation for $S_{std}$.
	Furthermore $S_{std}$ is an invariant subspace under the $\pi$-rotation about $z$-axis in $\mathbb{C}^2$ by definition.
	
	Note that the 4 regions in $\R^3$ separated by $S_{std}$ and $xy$-plane are
	cyclically related by $\pi/2$ rotation $Rot^z_{\pi/2}$.
	Moreover, it is not hard to check that the gluing map $\phi$ defined in \S\ref{subsec:singularconnectedsum}  is nothing but a restriction of $Rot^0_\pi$ to $S_{std}$.
	
	On the other hand, $Rot^0_{\pi/2}$ changes the roles of $z$ and $w$ up to $\pi/2$-rotation on $w$. Therefore the unit circle $S^1_{xy}\subset\R^3$ be mapped to the $z$-axis in $\R^3$, and the standard surfaces $S_{std}$ and $D_{std}$ correspond to noncompact surfaces $\widehat S_{std}$ and $\widehat D_{std}$, called {\em dual surfaces}, in $\R^3$ as depicted in Figure~\ref{fig:spheredual}. 
	Moreover, they correspond to when $z$ is purely imaginary, and when $z$ is real and positive, respectively.
	Hence for given $L\in\R^3$ with a singular point $p=\mathbf{0}$, the left normal form is obtained by $Rot^0_{\pi/2}(L)$.
	Similarly, the right normal form essentially comes from $Rot^0_{-\pi/2}(L)$, but the orientation does not match. Hence by compositing $Rot^w_\pi$, we have the exact right normal form, and the gluing map $\phi$ now becomes $Rot^z_{\pi}$ instead of $Rot^0_{\pi}$.
	
	\begin{figure}[ht]
		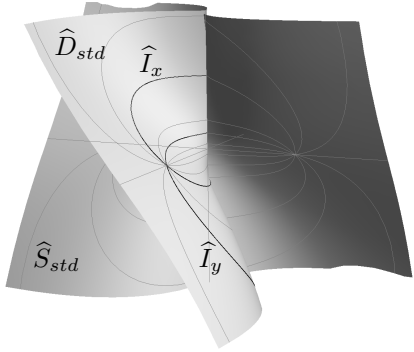
		\caption{Dual surfaces of $S_{std}$ and $D_{std}$}
		\label{fig:spheredual}
	\end{figure}
	
	In summary, the front projections of $Rot^0_{\pi/2}(L)$ and $Rot^w_\pi\circ Rot^0_{-\pi/2}(L)$ give us the left and right normal forms, respectively, and the gluing map $\phi$ is just $\pi$-rotation about $\widehat S^1_{std}=z\text{-axis}$.
	This is the justification for the diagrammatic definition for the singular connected sum.

\end{document}